\documentclass[12pt,reqno]{amsart}
\usepackage[a4paper]{geometry}
\usepackage{amssymb,amsmath,amsthm,enumerate,bbm}
\DeclareMathOperator{\Ran}{Ran}

\DeclareMathOperator{\Ker}{Ker}

\DeclareMathOperator{\diag}{diag}

\newcommand{\BMOA}{BMOA}

\renewcommand\Im{\hbox{{\rm Im}}\,}
\renewcommand\Re{\hbox{{\rm Re}}\,}
\newcommand{\abs}[1]{\lvert#1\rvert}
\newcommand{\Abs}[1]{\left\lvert#1\right\rvert}
\newcommand{\norm}[1]{\lVert#1\rVert}
\newcommand{\Norm}[1]{\left\lVert#1\right\rVert}
\newcommand{\jap}[1]{\langle#1\rangle}

\newcommand{\bbT}{{\mathbb T}}
\newcommand{\bbP}{{\mathbb P}}
\newcommand{\bbR}{{\mathbb R}}
\newcommand{\bbC}{{\mathbb C}}
\newcommand{\bbD}{{\mathbb D}}
\newcommand{\bbN}{{\mathbb N}}

\newcommand{\calH}{{\mathcal H}}

\newcommand{\calA}{\mathcal{A}}

\newcommand{\e}{\varepsilon}

\newcommand{\bg}{{\mathbf{g}}}
\newcommand{\bp}{{\mathbf{p}}}
\newcommand{\bbf}{{\mathbf{f}}}

\numberwithin{equation}{section}

\theoremstyle{plain}
\newtheorem{theorem}{\bf Theorem}[section]
\newtheorem*{theorem*}{Theorem 1.1$'$}
\newtheorem{lemma}[theorem]{\bf Lemma}
\newtheorem{proposition}[theorem]{\bf Proposition}
\newtheorem{corollary}[theorem]{\bf Corollary}

\theoremstyle{definition}

\newtheorem*{definition*}{\bf Definition}

\theoremstyle{remark}
\newtheorem*{remark*}{\bf Remark}
\newtheorem{remark}[theorem]{\bf Remark}

\newcommand{\wt}{\widetilde}

\newcommand{\eps}{\varepsilon}
\newcommand{\1}{\mathbbm{1}}

\begin{document}

\title[Inverse problem for Hankel operators]{An inverse problem for  Hankel operators  and  turbulent solutions of the cubic Szeg\H{o} equation on the line}

\author{Patrick G\'erard}
\address{Universit\'e Paris-Saclay, Laboratoire de Math\'ematiques d'Orsay, CNRS, UMR 8628}
\email{patrick.gerard@universite-paris-saclay.fr}

\author{Alexander Pushnitski}
\address{Department of Mathematics, King's College London, Strand, London, WC2R~2LS, U.K.}
\email{alexander.pushnitski@kcl.ac.uk}


\keywords{Hankel operators, inverse spectral problem, Hardy space, model spaces}

\date{8 February 2022}

\begin{abstract}
We construct inverse spectral theory for finite rank Hankel operators acting on the Hardy space of the upper half-plane. A particular feature of our theory is that we completely characterise the set of spectral data. As an application of this theory, we prove the genericity of turbulent solutions of the cubic Szeg\H{o} equation on the real line.
\end{abstract}

\maketitle

\setcounter{tocdepth}{1}
\tableofcontents

\section{Introduction}\label{sec.z}

\subsection{The cubic Szeg\H{o} equation and Hankel operators}

Let $H^2(\bbC_+)$  be the standard Hardy class in the upper half-plane, with the inner product (linear in the first factor and anti-linear in the second factor) denoted by 
$\jap{\cdot,\cdot}$. 
We will routinely identify functions $f\in H^2(\bbC_+)$ with their 
boundary values on the real line, and with this identification 
one has $H^2(\bbC_+)\subset L^2(\bbR)$ as a closed subspace, $\jap{\cdot,\cdot}$
being the restriction of the $L^2$ inner product. 
Let $\bbP_+$ be the orthogonal projection onto $H^2(\bbC_+)$ in $L^2(\bbR)$.

In \cite{OP1,OP2} Pocovnicu,
by analogy with the unit circle case \cite{GG1,GG2} (which will be discussed later in Section~\ref{sec.uc}), 
introduced and studied the \emph{cubic Szeg\H{o} equation} 
\begin{equation}
i\frac{\partial u}{\partial t}
=
\bbP_+(\abs{u}^2u), \quad u=u(x,t), \quad x,t\in \bbR.
\label{z1}
\end{equation}
Here,  for every $t\in \bbR$,  $u(\cdot,t)$
belongs to a suitable Sobolev class of functions in $H^2(\bbC_+)$. 
Following the strategy of the unit circle case \cite{GG1,GG2}, it was proven
in \cite{OP1,OP2} that this equation is completely integrable and possesses a Lax pair. 
The Lax pair involves the anti-linear Hankel operator $H_u$
on $H^2(\bbC_+)$,  defined by 
\begin{equation}
H_u f=\bbP_+(u\overline{f}), \quad f\in H^2(\bbC_+);
\label{z2}
\end{equation}
notice that the conjugate of $f$ is always taken on the real line.
It is well known that the boundedness of $H_u$ is equivalent to the inclusion $u\in\BMOA(\bbR)$. 
By a version of Kronecker's theorem, the rationality of $u$ is equivalent to $H_u$ being finite rank. 
We refer to  \cite{Peller} for the background on Hankel operators. 
Observe that while $H_u$ is anti-linear, the square $H_u^2$ is linear (and positive semi-definite). 
We will say that $\lambda>0$ is a \emph{singular value} of $H_u$, if the corresponding
\emph{Schmidt subspace} 
$$
E_{H_u}(\lambda)=\Ker(H_u^2-\lambda^2 I)
$$
is non-trivial: $E_{H_u}(\lambda)\not=\{0\}$. 
The Lax pair formulation for \eqref{z1} ensures, in particular, that all singular values 
of $H_u$ are integrals of motion of the Szeg\H{o} equation. 
In order to solve the Cauchy problem for the Szeg\H{o} equation, one must
develop a suitable version of direct and inverse spectral theory for the Hankel
operator $H_u$. 
In \cite{OP2}, this programme was completely achieved
when the symbol is rational and all singular values are simple 
(i.e. all Schmidt subspaces $E_{H_u}(\lambda)$ are one-dimensional).

The purpose of this paper is to extend the spectral analysis of \cite{OP2}
to the case of multiplicities, i.e. to the case when the symbol is rational but the dimensions
of the subspaces  $E_{H_u}(\lambda)$ may be $>1$. 
This requires a more detailed analysis of the structure
of these subspaces and of the action of $H_u$ on them. 
Such analysis was performed in \cite{GP2} and is recalled later on in this introduction. 

As an application of our spectral analysis, we  prove the genericity of turbulent solutions of the cubic Szeg\H{o} equation on the line.  In \cite{OP1,OP2} it was proved that, for every $s\geq 1/2$,  the initial value problem for \eqref{z1} is globally wellposed on the intersection $W^{s,2}(\bbC_+)$ of $H^2(\bbC_+)$ with the Sobolev space $W^{s,2}(\bbR )$ on the line. Though the trajectories are bounded in $W^{1/2,2}(\bbC_+)$ because of some conservation laws, they may not be bounded in $W^{s,2}(\bbC_+)$ if $s>1/2$. 
We shall call {\it turbulent}  a solution of 
\eqref{z1}  with an unbounded trajectory in $W^{s,2}(\bbC_+)$ for some $s>1/2$.  An example of a turbulent solution is provided in \cite{OP2} as a rational solution with two poles such that the associated Hankel operator has a singular value of multiplicity $2$. Using our spectral analysis, we are able to find many more such turbulent rational solutions, leading to the following result.
\begin{theorem}\label{genericgrowth}
There exists a dense $G_\delta $ subset $G$ of $W^{1,2}(\bbC_+)$ such any solution $u$ of the cubic Szeg\H{o} equation with initial datum in $G$ satisfies
$$\int_1^{+\infty}\frac{\norm{\partial_xu(t)}_{L^2}}{t^2}\, dt =+\infty \ .$$
\end{theorem}
In the above statement,  the regularity exponent $1$ in $W^{1,2}$ is probably not essential but it is  technically easier to handle.

We close this paragraph by some comments about the phenomenon of turbulent solutions for  Hamiltonian equations, which  has been actively studied by mathematicians in the last two decades.  Bourgain \cite{Bou00} asked whether there is a solution of the cubic defocusing nonlinear Schr\"odinger equation on the two-dimensional torus $\bbT^2$ with initial data $u_0\in W^{s,2}(\bbT^2)$, $s>1$, such that
$$\limsup_{t\to\infty}\|u(t)\|_{W^{s,2}}=\infty.$$
There is still no complete answer to this question, despite a first partial result in this direction  by Colliander--Keel--Staffilani--Takaoka--Tao  \cite{CKSTT10}.  Using this approach, Hani \cite{Hani} proved the existence of  turbulent solutions  for a totally resonant version of cubic NLS on $\bbT^2$, and Hani--Pausader--Tzvetkov--Visciglia \cite{HPTV} established the first --- and, still at this time, the unique --- example of  turbulent NLS solution,  in the case of the cubic Schr\"odinger equation on the cylinder $\bbT^2\times \bbR$. In the direction of the growth of Sobolev norms for the nonlinear Schr\"odinger equation, let us also mention the works of Guardia \cite{Guardia2014},  Guardia--Kaloshin \cite{GK2015},  Haus--Procesi \cite{HP},  Guardia--Haus--Procesi \cite{GHP},  and more recently the work of Guardia--Hani--Haus--Maspero--Procesi \cite{GHHMP2018}, which uses the integrable structure of the defocusing cubic nonlinear Schr\"odinger equation on the one dimensional torus.
\\
The phenomenon of growth of Sobolev norms also occurs for two-dimensional incompressible Euler equations: the sharp double exponentially growing vorticity gradient on the disk was constructed by Kiselev--\v{S}verak \cite{KS14} and the existence of exponentially growing vorticity gradient solutions on the torus was shown by Zlato\v{s} \cite{Zla15}. It also has been recently observed by Schwinte and Thomann \cite{Thom} for a system of two lowest Landau level equations. \\
At this stage, observe that the above results provide examples of turbulent solutions {\it without establishing their genericity}. In fact, as far as we know, the only equation where genericity of turbulent solutions has been proved before Theorem~\ref{genericgrowth} is  the cubic Szeg\H{o} equation on the circle \cite{GGAst,GTor}.

\subsection{Model spaces and isometric multipliers}\label{sec.z2}
Before recalling relevant results from \cite{GP2}, we need to 
talk about model spaces in $H^2(\bbC_+)$ and isometric multipliers on them. Let $\theta$ be an inner function in the upper half-plane (i.e. $\theta\in H^\infty(\bbC_+)$ 
and $\abs{\theta(x)}=1$ for a.e. $x\in \bbR$). 
We will only be concerned with the case when $\theta$ is a finite Blaschke product, i.e. 
\begin{equation}
\theta(x)=e^{i\alpha}\prod_{j=1}^m \frac{x-a_j}{x-\overline{a_j}}, 
\quad x\in\bbC_+,
\label{z10a}
\end{equation}
where $e^{i\alpha}$ is a unimodular complex number and the parameters $a_j$ satisfy $\Im a_j>0$.  

The model space $K_\theta\subset H^2(\bbC_+)$ is defined by 
$$
K_\theta=H^2(\bbC_+)\cap(\theta H^2(\bbC_+))^\perp;
$$
here 
$$
\theta H^2(\bbC_+)=\{\theta f: f\in H^2(\bbC_+)\}
$$
and  the orthogonal complement is taken in $L^2(\bbR)$. It is clear that $f\in H^2(\bbC_+)$ belongs to $K_\theta$ if and only if $\theta\overline{f}\in H^2(\bbC_+)$.
When $\theta$ is a finite Blaschke product \eqref{z10a}, the model space $K_\theta$ is finite-dimensional and can be explicitly described by 
$$
K_\theta=\left\{\frac{p(z)}{q(z)}, \quad \deg p\leq m-1\right\}, \quad
q(z)=\prod_{j=1}^m (x-\overline{a_j}), 
$$
where $p$ represents any polynomial of degree $\leq m-1$ (see e.g. \cite[Corollary 5.18]{GMR}). In particular, all elements of $K_\theta$ are rational functions.

Let $S(\tau)$  be the semigroup on $H^2(\bbC_+)$ defined by
\begin{equation}
(S(\tau)f)(x)=e^{i\tau x}f(x), \quad x\in\bbC_+, \quad \tau>0.
\label{z11}
\end{equation}
Further, let $P_\theta$ be the orthogonal projection onto $K_\theta$, and let 
$S_\theta(\tau)$ be Beurling-Lax semigroup on $K_\theta$, defined by 
\begin{equation}
S_\theta(\tau)f=P_\theta(S(\tau)f), \quad f\in K_\theta, \quad \tau>0
\label{z12}
\end{equation}
(see Section~\ref{sec.a} for a more detailed discussion). 
This is a strongly continuous contractive semigroup and therefore, 
by the theory of \cite{HillePh}, it can be written as 
\begin{equation}
S_\theta(\tau)=e^{i\tau A_\theta}, \quad \tau>0,
\label{z2a}
\end{equation}
where $A_\theta$ is the infinitesimal generator of $S_\theta(\tau)$. 
For a general inner function $\theta$, the operator $A_\theta$ may be unbounded, but under our assumption of rationality of $\theta$  the operator $A_\theta$ is easily seen to be bounded and of finite rank. Furthermore, it is dissipative, i.e. $\Im \langle A_\theta f, f\rangle \geq0$ for every $f$ in $K_\theta$. 

Let $p$ be a holomorphic function in $\bbC_+$. 
It is called an \emph{isometric multiplier} on $K_\theta$, 
if for every $f\in K_\theta$ we have $pf\in H^2(\bbC_+)$ 
and 
$$
\norm{pf}=\norm{f}, \quad \forall f\in K_\theta.
$$
In this case we can consider the subspace
$$
pK_\theta=\{pf: f\in K_\theta\}.
$$
We note that the choice of the parameters $p$, $\theta$ in the representation $pK_\theta$ for this subspace is not 
unique. In fact, if $p$ (resp. $\wt p$) is an isometric multiplier on 
$K_{\theta}$ (resp. on $K_{\wt\theta}$), then (see e.g. \cite[Theorem 10]{Crofoot})
$$
pK_{\theta}=\wt pK_{\wt\theta}
$$
if and only if for some $\abs{w}<1$ and some unimodular constants $e^{i\alpha}$, $e^{i\beta}$
we have 
\begin{equation}
\wt\theta=e^{i\alpha}\frac{w-\theta}{1-\overline{w}\theta}, 
\quad
\wt p=e^{i\beta}\frac{1-\overline{w}\theta}{\sqrt{1-\abs{w}^2}}p.
\label{z4}
\end{equation}
The transformation $\theta\mapsto \wt \theta$, $p\mapsto \wt p$ is called 
a \emph{Frostman shift}. 

\subsection{The Schmidt subspaces of Hankel operators}
The relevance of these objects to Hankel operators transpires from 
the following result:

\begin{theorem}\cite{GP2}\label{thm.a1}
Let $u\in \BMOA(\bbR)$ and let  $\lambda>0$ be a singular value of $H_u$. 
Then there exists 
an inner function $\psi$ and an isometric multiplier $p$ on $K_\psi$ such that 
the Schmidt subspace $E_{H_u}(\lambda)$ is represented as
$$
E_{H_u}(\lambda)=pK_\psi.
$$
Moreover,  there exists a unimodular constant $e^{i\alpha}$ such that
the action of $H_u$ on $E_{H_u}(\lambda)$ is given by 
\begin{equation}
H_u(ph)=\lambda e^{i\alpha}pH_\psi h, \quad h\in K_\psi.
\label{a16}
\end{equation}
In particular, by normalising $p$ and $\psi$ suitably (see \eqref{z4}), one can always achieve 
$e^{i\alpha}=1$.
\end{theorem}
We note that $H_\psi$ acts on $K_\psi$ in a simple explicit way:
$$
H_\psi h=\psi\overline{h}, \quad h\in K_\psi,
$$
because $\psi\overline{h}\in H^2(\bbC_+)$  and so $\bbP_+(\psi\overline{h})=\psi\overline{h}$.

\subsection{Direct and inverse spectral problems for Hankel operators with rational symbols}
Let $u$ be a rational symbol, analytic in the upper half-plane. We normalise $u$ so that $u(\infty)=0$ (subtracting a constant from $u$ does not change the Hankel operator $H_u$). It follows that $u\in H^2(\bbC_+)$.

First note that we have the commutation relation 
\begin{equation}
S(\tau)^*H_u=H_u S(\tau), \quad \tau>0.
\label{z3}
\end{equation}
In fact, $H_u$ is a Hankel operator if and only if it satisfies this relation.
It follows from \eqref{z3} that the kernel of $H_u$ is an invariant subspace for $S(\tau)$
and therefore, by the Beurling-Lax theorem \cite{Lax},
$$
\text{either }\quad
\overline{\Ran H_u}=H^2(\bbC_+)
\quad \text{ or }\quad
\overline{\Ran H_u}=K_\theta
$$
for some inner function $\theta$. 
For rational symbols $u$ the range of $H_u$ is finite-dimensional, and so we have the second possibility here with some rational inner function $\theta$. Further, the range of $H_u$ is closed, so finally we have 
$$
\Ran H_u=K_\theta
$$
with some finite Blaschke product $\theta$. It will be convenient to normalise $\theta$ so that $\theta(\infty)=1$. Let $A_\theta$ be the corresponding infinitesimal generator as in \eqref{z2a}.

Let us denote the singular values of $H_u$ by $\lambda_1>\lambda_2>\dots>\lambda_N>0$; 
to each of these singular values there corresponds a Schmidt subspace $E_{H_u}(\lambda_n)$, 
which may have arbitrary finite dimension. According to Theorem~\ref{thm.a1}, 
we have
\begin{equation}
E_{H_u}(\lambda_n)=p_n K_{\psi_n}, \quad n=1,\dots,N,
\label{z10}
\end{equation}
where $\psi_1,\dots,\psi_N$ are finite Blaschke products in $\bbC_+$ 
and each $p_n$ is an isometric multiplier on $K_{\psi_n}$ (in fact, each $p_n$ is a rational function, see Section~\ref{sec.c2}). 
We will normalise each $\psi_n$ so that $\psi_n(\infty)=1$.

The ``direct spectral problem'' for $H_u$ is given by the following theorem.
\begin{theorem}\label{thm.z2}
Let $u$ be a rational function analytic in the closed upper half-plane
and going to zero at infinity. Then
\begin{itemize}
\item
For all $n$, the vector $u\in H^2(\bbC_+)$ is not orthogonal to $E_{H_u}(\lambda_n)$. 
\item
Denoting by $u_n$ the orthogonal projection of $u$ onto $E_{H_u}(\lambda_n)$, 
we have
$$
H_u u_n=\lambda_n e^{i\varphi_n} u_n
$$
for some unimodular constants $e^{i\varphi_n}$. 
\item
The numbers 
$$
\omega_n:=\jap{A_\theta u_n,u_n}, \quad n=1,\dots,N
$$
have strictly positive imaginary parts. 
\end{itemize}
\end{theorem}

Next, we introduce the spectral data corresponding to a rational symbol $u$: 
\begin{equation}
\biggl(\{\lambda_n\}_{n=1}^N, \{\psi_n\}_{n=1}^N, \{e^{i\varphi_n}\}_{n=1}^N, 
\{\omega_n\}_{n=1}^N\biggr). 
\label{z8}
\end{equation}
Here $\lambda_n$, $e^{i\varphi_n}$ and $\omega_n$ are defined in the previous theorem and $\psi_n$ are defined in \eqref{z10}. According to the discussion of Section~\ref{z2}, the inner functions $\psi_n$ are defined only up to Frostman shifts. So it would be more precise to say that the spectral data contains the orbits of $\psi_n$ under all Frostman shifts, but for notational convenience we will be talking about representatives $\psi_n$ of these orbits. Moreover, it will turn out that our solution to the inverse problem is independent of the choice of a representative.

Now we can state, somewhat informally, one of our main results; the precise statements are Theorems~\ref{thm.uniq} and \ref{thm.c2} below. 
\begin{theorem}\label{thm.z3}
\begin{enumerate}[\rm (i)]
\item
Uniqueness: 
a rational symbol $u$ is uniquely determined by the spectral data \eqref{z8}.
In fact, $u$ is given by an explicit formula \eqref{b25} below.
\item
Surjectivity:
let $N\in \bbN$ and let 
\begin{itemize}
\item
$\lambda_1>\dots>\lambda_N>0$ be any positive real numbers, 
\item
$\psi_1,\dots\psi_N$ be any finite Blaschke products,
\item
$e^{i\varphi_1},\dots,e^{i\varphi_N}$ be any unimodular complex numbers,
\item
$\omega_1,\dots,\omega_N$ be any complex numbers with positive imaginary parts. 
\end{itemize}
Then there exists a rational symbol $u$ such that $H_u$ corresponds to the 
spectral data \eqref{z8}. 
\end{enumerate}
\end{theorem}

If all singular values $\lambda_n$ are simple, we recover the result of Pocovnicu \cite{OP2}. 
In this case the spectral data do not contain inner functions $\psi_n$. 
Let us explain this. 
If $\lambda_n$ is simple, then $\psi_n$ is a single Blaschke factor
$$
\psi_n(x)=\frac{x-a_n}{x-\overline{a_n}}, \quad a_n\in \bbC_+.
$$
By using a Frostman shift, $a_n$ can be changed to any given number in $\bbC_+$. 
In other words, in this case the orbit of $\psi_n$ by the action of Frostman shifts 
consists of all single Blaschke factors
and so this orbit does not contain any information apart from the fact that the dimension of the corresponding Schmidt subspace is one. 
In Section~\ref{sec.5.9} we give a more detailed comparison with the spectral data of \cite{OP2}.

Observe that the isometric multipliers $p_n$ are not part of our spectral data; in fact, they can be explicitly determined from the spectral data, see Section~\ref{sec.b4}. 

\subsection{The unit circle case}\label{sec.uc}
The Szeg\H{o} equation \eqref{z1}
was originally introduced in  \cite{GG1,GG2}, 
where it was considered for functions $u$ defined on the unit circle; 
$u$ was assumed to be in a suitable Sobolev subspace of the Hardy space $H^2(\bbD)$, where $\bbD$ is the unit disk.
In this context, $\bbP_+$ becomes the orthogonal projection in $L^2(\bbT)$ onto 
$H^2(\bbD)$, often called the Szeg\H{o} projection (hence the name for the equation). 
In fact, \cite{GG1,GG2} provided the blueprint for the study of the Lax pair structure of \cite{OP1,OP2}. 
More precisely, in the unit circle case the Szeg\H{o} equation is completely integrable and possesses a Lax pair, 
which involves a Hankel operator  $H_u$ in $H^2(\bbD)$. 
Solving the equation reduces to a solution of a direct and inverse spectral problem for $H_u$. 
Despite many similarities, we would like to stress
some important differences between the unit circle and the real line cases:
\begin{itemize}
\item
The choice of the spectral data in the unit circle case is very different. 
It involves introducing an auxiliary Hankel operator (denoted by $K_u$ in \cite{GG2})
and looking at its singular values and Schmidt subspaces. 
\item
The property $u\not\perp E_{H_u}(\lambda_n)$ of Theorem~\ref{thm.z2} 
is \emph{false} in the unit circle case! 
Roughly speaking, in the unit circle case for about half of the singular values (the ones termed \emph{$K$-dominant} in \cite{GG2})
we have $u\perp E_{H_u}(\lambda_n)$. 
\end{itemize}

Our proof of the second part of Theorem~\ref{thm.z3} is strongly inspired 
by \cite{GP} where a new ``algebraic'' approach to the inverse spectral problem
on the unit circle was developed. 

\subsection{The structure of the paper}
In Section~\ref{sec.a}, we recall various well-known facts from the theory of 
model spaces, focussing on $K_\theta$ with rational $\theta$. 
In Section~\ref{sec.a00}, we consider the set-up
$$
pK_\psi\subset K_\theta,
$$
where $p$ is an isometric multiplier on $K_\psi$ 
and derive identities relating the infinitesimal generators $A_\theta$
and $A_\psi$. 
Although model spaces is a well studied subject, some of our results here appear to be new. 

In Section~\ref{sec.aa}, we introduce Hankel operators and prove Theorem~\ref{thm.z2}. 
The key ingredient here is the commutation relation (recall that $\theta(\infty)=1$)
$$
A_\theta H_u^2-H_u^2 A_\theta
=\frac{i}{2\pi} \langle \cdot, H_u u\rangle (1-\theta)-\frac{i}{2\pi}\langle \cdot,u\rangle u\, .
$$
In Section~\ref{sec.b}, we prove the uniqueness part of Theorem~\ref{thm.z3}
by giving an explicit expression for $u(x)$ in terms of the spectral data. 
Our starting point is the formula
$$
u(x)=\jap{(A_\theta^*-x)^{-1}u,1-\theta}, \quad \Im x>0,
$$ 
which follows directly from the fact that $1-\theta$ is the ``reproducing
kernel of $K_\theta$ at infinity'' (see Lemma~\ref{lma.a1b} below). 
We then consider the action of $A_\theta$ on the model space 
$K_\theta=\Ran H_u$, represented as the orthogonal sum
\begin{equation}
K_\theta=\bigoplus_{n=1}^N p_n K_{\psi_n}.
\label{z9}
\end{equation}
We use the results of Section~\ref{sec.a00} and the above commutation relation 
for $A_\theta$ to show that 
$A_\theta$ has a rather special block-matrix structure in this
representation. Using this block-matrix structure, we express the resolvent $(A_\theta^*-x)^{-1}$ in terms of our spectral data. 

In Sections~\ref{sec.c} and ~\ref{proof6.2}, we prove the last part of Theorem~\ref{thm.z3}. 
Here we use the ``algebraic'' approach of \cite{GP}. Namely, we take $u$ given 
by the explicit expression established at the previous step of the proof and check 
directly that the corresponding operator $H_u$ has the ``correct'' spectral data. 
An important step in the proof is checking that the functions $p_n$, given by 
certain explicit matrix formulas, are isometric multipliers on $K_{\psi_n}$. 
Here we are guided by intuition coming from Sarason's work \cite{Sarason}, 
which gives a general representation formula for all isometric multipliers on 
a given model space. Lemma~\ref{lma.c11} provides a partial extension of this 
formula to the matrix case.

Section~\ref{evol}  is devoted to describing the evolution of the spectral data \eqref{z8} of a solution of the cubic Szeg\H{o} equation \eqref{z1}. Finally, in Section~\ref{growth}, we combine the previous results to prove Theorem~\ref{genericgrowth}.

\section{Model spaces}\label{sec.a}

Almost all of this section is either well known or folklore; see e.g. the monograph \cite{GMR}.  Some of the facts that we need are easy to find in the literature but for the case of the Hardy spaces on the unit disk rather than the upper half-plane. In any case, for completeness we present all necessary statements with simple proofs. 

\subsection{Model spaces}\label{sec.a0}
In what follows we denote by $H^2$ the standard Hardy class in the upper half-plane $\bbC_+$.
Let $v_\zeta$, $\zeta\in \bbC_+$, be the reproducing kernel in $H^2$, 
$$
v_\zeta(x)=\frac1{2\pi i}\frac{1}{\overline{\zeta}-x},
$$
so that $f(\zeta )=\jap{f,v_\zeta}$ for every $f\in H^2$.
Let $\theta$ be a finite Blaschke product in $\bbC_+$ and let $K_\theta$ be the corresponding model space 
$$
K_\theta=H^2\cap (\theta H^2)^\perp.
$$
Further, let $P_\theta$ the orthogonal projection in $H^2$ onto $K_\theta$; it is straightforward to see (see e.g. \cite[Proposition 5.14]{GMR}) that  $P_{\theta}$  is given by 
\begin{equation}
P_\theta: f\mapsto f-\theta\bbP_+(\overline{\theta}f), \quad f\in H^2.
\label{a0b}
\end{equation}
Let $S(\tau)$ and $S_\theta(\tau)$ be the semigroups \eqref{z11}, \eqref{z12}, and let $A_\theta$ be the infinitesimal generator of $S_\theta(\tau)$ as in \eqref{z2a}. 
For $\zeta\in \bbC_+$ the following resolvents
are well-defined and bounded on $K_\theta$: 
\begin{align*}
(A_\theta-\overline{\zeta})^{-1}&=-i\int_0^\infty 
S_\theta(\lambda)e^{-i\lambda\overline{\zeta}}d\lambda,
\\
(A_\theta^*-\zeta)^{-1}&=i\int_0^\infty 
S_\theta(\lambda)^*e^{i\lambda\zeta}d\lambda, 
\end{align*}
see e.g. \cite{HillePh}.
Using the definition of $S_\theta$ and computing the integrals, we see that these resolvents can be expressed as
\begin{align}
(A_\theta-\overline{\zeta})^{-1}f&
=
-2\pi i P_\theta(fv_\zeta),
\label{a0c}
\\
(A_\theta^*-\zeta)^{-1}f&
=
2\pi i\bbP_+(f\overline{v_\zeta})
=
\frac{f(x)-f(\zeta)}{x-\zeta}.
\label{a0d}
\end{align}

\begin{lemma}\label{lma.nsa}
The operator $A_\theta$ is completely non-selfadjoint, i.e. there is no 
non-trivial subspace $N\subset K_\theta$, invariant for $A_\theta$ where $A_\theta$ is 
self-adjoint. 
\end{lemma}
\begin{proof}
From the definition of $S(\tau)$ it follows that 
$$
\lim_{\tau\to\infty}\norm{S (\tau)^*f}=0
$$
for any $f\in H^2$. Since $\theta H^2$ is invariant under $S(\tau)$, the model space $K_\theta$ is invariant under $S(\tau)^*$ and therefore $S_\theta(\tau)^*f=S(\tau)^*f$ for any $f\in K_\theta$. Thus we also have 
$$
\lim_{\tau\to\infty}\norm{S_\theta(\tau)^*f}=0
$$
for any $f\in K_\theta$.
It follows that $S_\theta(\tau)^*$ is completely non-unitary on $K_\theta$, i.e. there 
is no non-trivial subspace $N\subset K_\theta$, invariant for $S_\theta(\tau)^*$ for all $\tau>0$
and such that  $S_\theta(\tau)^*$ is unitary on $N$ for all $\tau>0$.
From here we get the claim.  
\end{proof}

We write $\theta(\infty)=\lim\limits_{\abs{x}\to\infty}\theta(x)$; clearly we have $\abs{\theta(\infty)}=1$. In what follows, we normalise $\theta$ so that $\theta(\infty)=1$.

\begin{lemma}\label{lma.a1a}
Let $\theta$ be a finite Blaschke product and $\theta(\infty)=1$. 
Then $1-\theta\in K_\theta$. 
\end{lemma}
\begin{proof}
For $h\in H^2$, we have
$$
\jap{\theta h,1-\theta}
=
\jap{h,\overline{\theta}(1-\theta)}
=
\jap{h,\overline{\theta}-1}
=0.
$$
It follows that $1-\theta$ is orthogonal to $\theta H^2$ 
and so it belongs to $K_\theta$. 
\end{proof}

\begin{lemma}\label{lma.a7}
Let $\theta$ be a finite Blaschke product and $\theta(\infty)=1$. 
Then for any $f\in K_\theta$ and for any $\zeta\in \bbC_+$ we have
\begin{align}
f(\zeta)&=\frac1{2\pi i}\jap{f,(A_\theta -\overline{\zeta})^{-1}(1-\theta)},
\label{a17}
\\
(H_\theta f)(\zeta)&=-\frac1{2\pi i}\jap{(A_\theta^*-\zeta)^{-1}(1-\theta),f}.
\notag
\end{align}
\end{lemma}
\begin{proof}
We have 
$$
f(\zeta)=\jap{f,v_\zeta}=\jap{f,P_\theta v_\zeta}
=\jap{f,P_\theta(v_\zeta(1-\theta))}, 
$$
because $P_\theta(\theta v_\zeta)=0$. 
By \eqref{a0c}, 
$$
\jap{f,P_\theta (v_\zeta(1-\theta))}
=
\frac1{2\pi i}
\jap{f,(A_\theta -\overline{\zeta})^{-1}(1-\theta)},
$$
which yields \eqref{a17}.
Similarly, 
\begin{multline*}
H_\theta f(\zeta)
=
\jap{\theta \overline{f},v_\zeta}
=
\jap{\theta \overline{v_\zeta},f}
\\
=
\jap{\bbP_+(\theta \overline{v_\zeta}),f}
=
-\jap{\bbP_+((1-\theta)\overline{v_\zeta}),f}
=
-\frac1{2\pi i}\jap{(A_\theta^*-\zeta)^{-1}(1-\theta),f},
\end{multline*}
where we have used \eqref{a0d} at the last step. 
\end{proof}

\begin{corollary}\label{cr.a7}
Let $\theta$ be a finite Blaschke product and $\theta(\infty)=1$. 
Then the linear span of each of the two sets 
$$
\{(A_\theta-\overline{\zeta})^{-1}(1-\theta)\}_{\zeta\in \bbC_+}, 
\quad
\{(A_\theta^*-\zeta)^{-1}(1-\theta)\}_{\zeta\in \bbC_+}
$$
is dense in $K_\theta$. 
\end{corollary}

\subsection{Behaviour at infinity} 
For a rational function $f$ with the  Laurent expansion 
$$
f(x)=a_0+\frac{a_{1}}{x}+\frac{a_{2}}{x^2}+\dots, \quad \abs{x}\to\infty
$$
at infinity, we will denote 
$$
\Lambda_1(f)=a_{1}, \quad \Lambda_2(f)=a_2.
$$

\begin{lemma}\label{lma.a1b}
Let $\theta$ be a finite Blaschke product and $\theta(\infty)=1$. 
For any $f\in K_\theta$, we have
$$
\Lambda_1(f)=-\frac1{2\pi i}\jap{f,1-\theta}, \quad 
\Lambda_2(f)=-\frac1{2\pi i}\jap{A_\theta^*f,1-\theta}, 
$$
and in particular, 
\begin{equation}
\Lambda_1(\theta)=\frac1{2\pi i}\norm{1-\theta}^2, \quad
\Lambda_2(\theta)=\frac1{2\pi i}\jap{A_\theta^*(1-\theta),1-\theta}.
\label{a6a}
\end{equation}
\end{lemma}
\begin{proof}
Since $A_\theta$ is bounded, we can expand the resolvent in \eqref{a17}, 
which yields
$$
f(x)=-\frac1{2\pi i}\frac1x \jap{f,1-\theta}-\frac1{2\pi i}\frac1{x^2}\jap{A_\theta^* f,1-\theta}+O(1/x^3)
$$
as $\abs{x}\to\infty$. 
This yields the required identities. 
\end{proof}

\subsection{Formulas for  $A_\theta$ and $A_\theta^*$}

\begin{lemma}\label{lma.a2}
Let $\theta$ be a finite Blaschke product and $\theta(\infty)=1$. 
Then the operator $A_\theta$ on $K_\theta$ satisfies the identities
\begin{align}
A_\theta f(x)&=xf(x)-\Lambda_1(f)\theta(x)=xf(x)+\frac1{2\pi i}\jap{f,1-\theta}\theta(x),
\label{a3}
\\
A_\theta^* f(x)&=xf(x)-\Lambda_1(f)=xf(x)+\frac1{2\pi i}\jap{f,1-\theta},
\label{a4}
\\
\Im A_\theta&=\frac1{4\pi}\jap{\cdot, 1-\theta}(1-\theta). 
\label{a5}
\end{align}
\end{lemma}
\begin{proof}
Let $A_\theta$ be the operator given by the right hand side of \eqref{a3}. 
We need to check that for all $f\in K_\theta$, 
\begin{equation}
\Norm{\tfrac{1}{i\tau}(S_\theta(\tau)-I)f-A_\theta f}\to0,
\quad 
\tau\to0_+\ .
\label{a7}
\end{equation}
First we observe that for all $\tau>0$ and any $h\in K_\theta$, 
\begin{equation}
\int_{-\infty}^\infty \frac{e^{i\tau x}-1}{i\tau x}\theta(x)\overline{h(x)}dx=0,
\label{a6}
\end{equation}
since $\tfrac{e^{i\tau x}-1}{i\tau x}\in H^2$ and $\theta\overline{h}\in H^2$. 
Using this, we compute
\begin{multline*}
\tfrac1{i\tau}\jap{(S_\theta(\tau)-I)f,h}
=
\tfrac1{i\tau}\jap{(S(\tau)-I)f,h}
=
\int_{-\infty}^\infty \frac{e^{i\tau x}-1}{i\tau x}xf(x)\overline{h(x)}dx
\\
=
\int_{-\infty}^\infty \frac{e^{i\tau x}-1}{i\tau x}(xf(x)-\Lambda_1(f)\theta(x))\overline{h(x)}dx
=
\int_{-\infty}^\infty \frac{e^{i\tau x}-1}{i\tau x}A_\theta f(x)\overline{h(x)}dx, 
\end{multline*}
and therefore
$$
\jap {\tfrac1{i\tau}(S_\theta(\tau)-I)f-A_\theta f,h}
=
\int_{-\infty}^\infty \biggl(\frac{e^{i\tau x}-1}{i\tau x}-1\biggr) A_\theta f(x) \overline{h(x)}dx.
$$
By Cauchy-Schwarz, 
$$
\Abs{\jap{\tfrac1{i\tau}(S_\theta(\tau)-I)f-A_\theta f,h}}^2
\leq
\norm{g}^2\int_{-\infty}^\infty \Abs{\frac{e^{i\tau x}-1}{i\tau x}-1}^2\abs{A_\theta f(x)}^2 dx,
$$
and the integral in the r.h.s. tends to zero as $\lambda\to0_+$ by dominated convergence. 
This yields \eqref{a7}. 

Similarly, let $A_\theta^*$ be the operator given by the r.h.s. of \eqref{a4}; we need to check that
$$
\Norm{\tfrac{1}{i\tau}(S(\tau)^*-I)f-A_\theta^* f}\to0,
\quad 
\tau\to0_+
$$
for all $f\in K_\theta$. 
This is achieved by following the same line of reasoning, with the only difference that instead of
\eqref{a6} we use the identity
$$
\int_{-\infty}^\infty \frac{e^{-i\tau x}-1}{i\tau x}\overline{h(x)}dx=0, \quad \tau>0, \quad g\in K_\theta,
$$
because both factors in the integral are complex conjugates of elements of $H^2$. 

Finally, \eqref{a5} follows by subtracting \eqref{a4} from \eqref{a3}. 
\end{proof}

\subsection{The action of Frostman shifts}\label{sec.c4}
Let $\psi$ be a finite Blaschke product. 
Here we compute the action of Frostman shifts \eqref{z4} 
on various quantities relevant to the subsequent analysis of inverse problems. 
First note that if we require that $\psi(\infty)=\wt\psi(\infty)=1$, this fixes the 
unimodular constant $e^{i\alpha}$ in \eqref{z4}, so we obtain 
\begin{equation}
 \wt\psi
=
\frac{1-\overline{w}}{w-1}\frac{w-\psi}{1-\overline{w}\psi}, \quad \abs{w}<1.
\label{b17a}
\end{equation}
Next, observe that the function $i\frac{1+\psi}{1-\psi}$ is a Herglotz function, i.e. it maps the upper half-plane into itself. As a rational Herglotz function, it admits the representation 
$$
i\frac{1+\psi(x)}{1-\psi(x)}=Ax+B+\sum_j \frac{c_j}{\alpha_j-x}, \quad \Im x>0,
$$
where $A\geq0$, $B\in \bbR$, $c_j\geq0$, $\alpha_j\in\bbR$ and the sum is finite; the points $\alpha_j$ are the solutions to the equation $\psi(x)=1$. Recalling that 
\begin{equation}
\norm{1-\psi}^2=2\pi i\Lambda_1(\psi)
\label{b20}
\end{equation}
and renormalising, we obtain the representation 
\begin{equation}
i\frac{\norm{1-\psi}^2}{4\pi}\frac{1+\psi(x)}{1-\psi(x)}=x+B+\sum_j \frac{c_j}{\alpha_j-x}, \quad \Im x>0,
\label{b20aa}
\end{equation}
with the same conditions on the parameters $B$, $c_j$, $\alpha_j$.

\begin{lemma}\label{lma.c7a}
Let $\psi$ be a finite Blaschke product with $\psi(\infty)=1$, and let $\wt \psi$ be as in \eqref{b17a}. Then 
\begin{align}
\norm{1-\wt \psi}^2
&=
\frac{1-\abs{w}^2}{\abs{1-w}^2}\norm{1-\psi}^2, 
\label{b18}
\\
i\frac{\norm{1-\wt \psi}^2}{4\pi}\frac{1+\wt \psi}{1-\wt\psi}
&=
i\frac{\norm{1-\psi}^2}{4\pi}\frac{1+\psi}{1-\psi}
+
\frac1{2\pi}\frac{\Im w}{\abs{1-w}^2}\norm{1-\psi}^2.
\label{b19}
\end{align}
\end{lemma}
\begin{proof}
Computing the asymptotics of $\wt\psi$ at infinity, comparing with the asymptotics of $\psi$ and using \eqref{b20},
we obtain \eqref{b18}. 
After this, the proof of \eqref{b19} is direct algebra.
\end{proof}

This lemma shows  that under the Frostman shift, only the constant $B$ in the representation \eqref{b20aa} changes. 
The next lemma gives more precise information about this constant. 

\begin{lemma}\label{lma.b17}
Let $\psi$ be a finite Blaschke product with $\psi(\infty)=1$. 
Then 
\begin{equation}
i
\frac{\norm{1-\psi}^2}{4\pi}\frac{1+\psi(x)}{1-\psi(x)}
=
x-\frac{\Re \jap{A_\psi(1-\psi),1-\psi}}{\norm{1-\psi}^2}
+O(1/x)
\label{b20a}
\end{equation}
as $\abs{x}\to\infty$. As a consequence, the function 
\begin{equation}
i\frac{\norm{1-\psi}^2}{4\pi}\frac{1+\psi(x)}{1-\psi(x)}+\frac{\Re \jap{A_\psi(1-\psi),1-\psi}}{\norm{1-\psi}^2}
\label{b20b}
\end{equation}
is invariant under the Frostman shifts  $\psi\mapsto \wt\psi$ as in \eqref{b17a}. 
\end{lemma}
\begin{proof}
Using \eqref{b20} and expanding the function at infinity, we get
\begin{align*}
i\frac{\norm{1-\psi}^2}{4\pi}\frac{1+\psi(x)}{1-\psi(x)}
&=
-\frac12 \Lambda_1(\psi)\frac{2+\frac{\Lambda_1(\psi)}{x}+O(1/x^2)}{-\frac{\Lambda_1(\psi)}{x}-\frac{\Lambda_2(\psi)}{x^2}+O(1/x^3)}
\\
&=
x\frac{1+\frac{\Lambda_1(\psi)}{2x}+O(1/x^2)}{1+\frac{\Lambda_2(\psi)}{x\Lambda_1(\psi)}+O(1/x^2)}
\\
&=x-\biggl(\frac{\Lambda_2(\psi)}{\Lambda_1(\psi)}-\frac{\Lambda_1(\psi)}{2}\biggr)
+O(1/x)
\end{align*}
as $\abs{x}\to\infty$. 
By \eqref{a6a},  \eqref{b20} and \eqref{a5}, the constant term in the above formula transforms as
$$
\frac{\Lambda_2(\psi)}{\Lambda_1(\psi)}-\frac{\Lambda_1(\psi)}{2}
=
\frac{\jap{A_\psi^*(1-\psi),1-\psi}}{\norm{1-\psi}^2}
-\frac1{4\pi i}\norm{1-\psi}^2
=
\frac{\Re \jap{A_\psi(1-\psi),1-\psi}}{\norm{1-\psi}^2},
$$
which yields \eqref{b20a}. By Lemma~\ref{lma.c7a}, the function \eqref{b20b} changes by a constant under the Frostman shift. 
By \eqref{b20a}, this constant is zero. 
\end{proof}


\section{Isometric multipliers on model spaces}\label{sec.a00}

In this section we consider the following scenario: $\theta$ is a finite Blaschke product,
$\psi$ is an inner function in $\bbC_+$ and $p$ is an isometric multiplier on $K_\psi$. We assume that $pK_\psi\subset K_\theta$ and derive some identities relating the infinitesimal generators $A_\theta$ and $A_\psi$. Although the results of this section are relatively straightforward, the set-up is rather special and hard to locate in the literature.

\subsection{Formula for the projection onto $pK_\psi$}

Let $\psi$ be a non-constant inner function in $\bbC_+$, and 
let $p$ be an isometric multiplier on $K_\psi$. 
Denote $M=pK_\psi$ and let $P_M$ be the orthogonal 
projection onto $M$. First we need a formula for $P_M$:

\begin{lemma}\label{lma.a2b}
Let $f\in H^2$ be such that $\overline{p}f\in L^2(\bbR)$. Then 
\begin{equation}
P_M f=p\bbP_+(\overline{p}f)-p\psi\bbP_+(\overline{p}\overline{\psi}f). 
\label{a9}
\end{equation}
\end{lemma}
\begin{proof}
Using formula \eqref{a0b} for the orthogonal projection $P_\psi$ onto $K_\psi$, we see that 
the r.h.s. of \eqref{a9} can be written as
$$
p\bbP_+(\overline{p}f)-p\psi\bbP_+(\overline{p}\overline{\psi}f)=p P_{\psi}\bbP_+(\overline{p}f). 
$$
It is clear that the r.h.s. here is in $pK_\psi$. 
It remains to check that its difference with $f$ is orthogonal to $ pK_\psi$. 
For $h\in K_\psi$, we have
$$
\jap{f-pP_{\psi}\bbP_+(\overline{p}f),ph}
=
\jap{f,ph}-\jap{P_{\psi}\bbP_+(\overline{p}f),h}
=
\jap{f,ph}-\jap{\overline{p}f,h}=0,
$$
as required. 
\end{proof}

\subsection{Projecting onto $M=pK_\psi$ in $K_\theta$}\label{sec.c2}

Here we work out formulas for projections of various functions onto $M=pK_\psi$ in $K_\theta$. 

\begin{lemma}\label{lma.a3a}
Let $\psi$ and $\theta$ be non-constant inner functions in $\bbC_+$ and 
let $p$ be an isometric multiplier on $K_\psi$. Assume that
$pK_\psi\subset K_\theta$ and that $\theta$ is a finite Blaschke product. 
Then both $\psi$ and $p$ are rational (in particular, $\psi$ is also a finite Blaschke product).  
Furthermore, 
\begin{equation}
p-p(\infty)\in K_\theta.
\label{b11a}
\end{equation} 
\end{lemma}
\begin{proof}
It is easy to compute the reproducing kernel of $K_\psi$ (cf. e.g. \cite[Section 5.5]{GMR}):
$$
F_\zeta(x):=P_{\psi}v_\zeta(x)
=
\frac{1}{2\pi i}
\frac{1-\overline{\psi(\zeta)}\psi(x)}{\overline{\zeta}-x}.
$$
Since $F_\zeta\in K_\psi$ and $pK_\psi\subset K_\theta$, we see that $pF_\zeta$ is rational, with $(pF_\zeta)(\infty )=0$. 
Multiplying by $(\overline{\zeta}-x)$, we find that the function 
$$
(1-\overline{\psi(\zeta)}\psi(x))p(x)
$$
is rational.
Since this is true for any $\zeta\in \bbC_+$ and $\psi$ is non-constant, 
we conclude that both $p$ and $\psi p$ are rational. 
Since $p$ is not identically zero, we finally conclude that $\psi$ is also rational. 

Let us prove \eqref{b11a}. 
Take some $h\in K_\psi$ with $\Lambda_1(h)\not=0$, for example, $h=1-\psi$ (if $\psi$ is normalised by $\psi(\infty)=1$), see \eqref{a6a}. 
By \eqref{a4}, we have
\begin{multline*}
A_\theta^*(ph)-pA_\psi^*h
=
xp(x)h(x)-\Lambda_1(ph)-p(x)(xh(x)-\Lambda_1(h))
\\
=
p(x)\Lambda_1(h)-\Lambda_1(ph)
=
(p(x)-p(\infty))\Lambda_1(h).
\end{multline*}
Here the left side is in $K_\theta$, and so the right side is also in $K_\theta$. 
\end{proof}

\begin{lemma}\label{lma.a3}
Let $\psi$ and $\theta$ be finite Blaschke products with $\psi(\infty)=\theta(\infty)=1$, and let $p$ be an isometric multiplier on $K_\psi$. Assume that $M=pK_\psi\subset K_\theta$. 
Then 
$$
P_M(1-\theta)=\overline{p(\infty)}p(1-\psi)
$$
and 
\begin{equation}
\norm{P_M(1-\theta)}=\abs{p(\infty)}\norm{1-\psi}.
\label{a11aaa}
\end{equation}
\end{lemma}
\begin{proof}
By formula \eqref{a9}, we have
$$
P_M(1-\theta)=p\bbP_+(\overline{p}(1-\theta))-p\psi\bbP_+(\overline{p}\overline{\psi}(1-\theta)).
$$
Let us compute the two terms in the right side. 
We have 
\begin{multline*}
\overline{p}(1-\theta)
=
(\overline{p}-\overline{p(\infty)})(1-\theta)+\overline{p(\infty)}(1-\theta)
\\
=
(\overline{p}-\overline{p(\infty)})-\theta(\overline{p}-\overline{p(\infty)})+\overline{p(\infty)}(1-\theta).
\end{multline*}
By \eqref{b11a}, the first term in the right hand side is 
anti-analytic and the second term is in $H^2$, so we get
\begin{equation}
\bbP_+(\overline{p}(1-\theta))
=
-\theta(\overline{p}-\overline{p(\infty)})+\overline{p(\infty)}(1-\theta)
=
\overline{p(\infty)}-\theta\overline{p}.
\label{a11a}
\end{equation}
Next, we have 
$$
\overline{p}\overline{\psi}(1-\theta)
=
\overline{p}(\overline{\psi}-1)-\theta\overline{p}(\overline{\psi}-1)+\overline{p}(1-\theta).
$$
Since $p(\psi-1)\in pK_\psi\subset K_\theta$, we have $\theta\overline{p}(\overline{\psi}-1)\in H^2$. 
Thus, using \eqref{a11a}, 
$$
\bbP_+(\overline{p}\overline{\psi}(1-\theta))
=
-\theta\overline{p}(\overline{\psi}-1)+\bbP_+(\overline{p}(1-\theta))
=
-\theta\overline{p}(\overline{\psi}-1)
+
\overline{p(\infty)}-\theta\overline{p}.
$$
Combining this, we obtain
\begin{multline*}
P_M(1-\theta)
=
p(\overline{p(\infty)}-\theta\overline{p})
-
p\psi(\overline{p(\infty)}-\theta\overline{p}-\theta\overline{p}(\overline{\psi}-1))
\\
=
(1-\psi)p(\overline{p(\infty)}-\theta\overline{p})
+
\theta\psi \abs{p}^2(\overline{\psi}-1)
=
\overline{p(\infty)}p(1-\psi),
\end{multline*}
as claimed. Computing the norms and using the isometricity of $p$, we obtain
\eqref{a11aaa}. 
\end{proof}

\begin{lemma}\label{lma.a3b}
Assume the hypothesis of the previous Lemma. 
Then 
$$
P_M(p-p(\infty))=\varkappa p(1-\psi)
$$
with some $\varkappa \in \bbC$. 
\end{lemma}

Unlike in the previous lemma, we don't have a 
direct argument for it, nor an expression for the constant $\varkappa$. 
Our proof requires two intermediate steps. 

\begin{lemma}\label{lma.a3c}
Assume the hypothesis of Lemma~\ref{lma.a3} and let $\zeta\in\bbC_+$. 
Then 
$$
P_M(A_\theta^*-\zeta)^{-1}(p(1-\psi))=cp(A_\psi^*-\zeta)^{-1}(1-\psi)
$$
with some $c\in \bbC$. 
\end{lemma}
\begin{proof}
The statement of the lemma is equivalent to the following one. 
Let $f\in pK_\psi$, $f\perp p(A_\psi^*-\zeta)^{-1}(1-\psi)$; then
\begin{equation}
f\perp (A_\theta^*-\zeta)^{-1}(p(1-\psi)). 
\label{a11b}
\end{equation}
Write $f=ph$, $h\in K_\psi$; then condition 
$f\perp p(A_\psi^*-\zeta)^{-1}(1-\psi)$ is equivalent to 
$h\perp (A_\psi^*-\zeta)^{-1}(1-\psi)$. 
We have 
$$
\jap{f,(A_\theta^*-\zeta)^{-1}(p(1-\psi))}
=
\jap{(A_\theta-\overline{\zeta})^{-1}(ph),p(1-\psi)},
$$
and, by \eqref{a0c},
$$
(A_\theta-\overline{\zeta})^{-1}(ph)
=
-2\pi i\, P_\theta(phv_\zeta)
=
- 2\pi i\, phv_\zeta+w, \quad w\in \theta H^2.
$$
It follows that 
\begin{equation}
\jap{(A_\theta-\overline{\zeta})^{-1}(ph),p(1-\psi)}
=
-2\pi i
\jap{phv_\zeta,p(1-\psi)}.
\label{a11c}
\end{equation}

We would like to use the isometricity of $p$ in the right side of \eqref{a11c}. 
In order to be able to do so, we must check that 
$hv_\zeta$ is in $K_\psi$. Let $g\in H^2$; consider 
$$
\jap{hv_\zeta,\psi g}
=
\jap{h,\overline{v_\zeta}\psi g}
=
\jap{h,\bbP_+(\overline{v_\zeta}\psi g)}.
$$
Further, we have 
\begin{align*}
\bbP_+(\overline{v_\zeta}\psi g)(x)
&=
\frac1{2\pi i}\frac{\psi(x)g(x)-\psi(\zeta)g(\zeta)}{x-\zeta}
\\
&=
\frac1{2\pi i}\psi(x)\frac{g(x)-g(\zeta)}{x-\zeta}
+\frac1{2\pi i}g(\zeta)\frac{(\psi(x)-1)-(\psi(\zeta)-1)}{x-\zeta}
\\
&=
\psi\wt w+\frac1{2\pi i}g(\zeta)(A_\psi^*-\zeta)^{-1}(\psi-1), 
\end{align*}
where $\wt w\in H^2$. 
Now
$$
\jap{h,\bbP_+(\overline{v_\zeta}\psi g)}
=
\jap{h,\psi\wt w}
-
\frac1{2\pi i}\overline{g(\zeta)}
\jap{h,(A_\psi^*-\zeta)^{-1}(1-\psi)}
=0,
$$
where the first term in the right side vanishes because $h\in K_\psi$ 
and the second term vanishes because of our condition on $h$. 
Thus, we have checked that $hv_\zeta\in K_\psi$. 

We come back to \eqref{a11c}: 
$$
\jap{phv_\zeta,p(1-\psi)}
=
\jap{hv_\zeta,1-\psi}
=
-(2\pi i)\Lambda_1(hv_\zeta),
$$
where we have used Lemma~\ref{lma.a1b} at the last step. 
Finally, we have $h(x)=O(1/x)$ and $v_\zeta(x)=O(1/x)$ at
infinity, and therefore $\Lambda_1(hv_\zeta)=0$. 
We have checked \eqref{a11b}; the proof is complete. 
\end{proof}

\begin{lemma}\label{lma.a3d}
Assume the hypothesis of Lemma~\ref{lma.a3} and let $\zeta\in\bbC_+$. 
Then 
\begin{equation}
P_M(A_\theta^*-\zeta)^{-1}(p-p(\infty))=c'p(A_\psi^*-\zeta)^{-1}(1-\psi)
\label{a11d}
\end{equation}
with some $c'\in \bbC$. 
\end{lemma}
\begin{proof}
We have, from \eqref{a0d}, 
\begin{multline*}
(A_\theta^*-\zeta)^{-1}(p(1-\psi))
=
\frac{p(x)(1-\psi(x))-p(\zeta)(1-\psi(\zeta))}{x-\zeta}
\\
=
p(x)\frac{(1-\psi(x))-(1-\psi(\zeta))}{x-\zeta}
+
(1-\psi(\zeta))\frac{(p(x)-p(\infty))-(p(\zeta)-p(\infty))}{x-\zeta}
\\
=
p(A_\psi^*-\zeta)^{-1}(1-\psi)
+
(1-\psi(\zeta))(A_\theta^*-\zeta)^{-1}(p-p(\infty)). 
\end{multline*}
Let us apply $P_M$ to both sides of this identity. 
Observing that $p(A_\psi^*-\zeta)^{-1}(1-\psi)\in pK_\psi$ and using the 
previous lemma, we obtain
$$
cp(A_\psi^*-\zeta)^{-1}(1-\psi)
=
p(A_\psi^*-\zeta)^{-1}(1-\psi)
+
(1-\psi(\zeta))P_M(A_\theta^*-\zeta)^{-1}(p-p(\infty)). 
$$
Rearranging, we obtain the required identity with 
$c'=(c-1)/(1-\psi(\zeta))$. 
\end{proof}

\begin{proof}[Proof of Lemma~\ref{lma.a3b}]
We have the strong convergence  
\begin{equation}
-\zeta(A_\theta^*-\zeta)^{-1}\to I, \quad \abs{\zeta}\to\infty
\label{a14}
\end{equation}
and the same is true for the resolvent of $A_\psi^*$. 
Applying this to  \eqref{a11d}, we obtain the required result. 
\end{proof}

\subsection{Relation between $A_\theta$ and $A_\psi$}

\begin{theorem}\label{lma.a4}
Let $\psi$ and $\theta$ be finite Blaschke products normalised so that $\psi(\infty)=\theta(\infty)=1$. 
Let $p$ be an isometric multiplier on $K_\psi$; assume that
$M=pK_\psi\subset K_\theta$. 
Then 
there exists a constant $c\in\bbC$ such that 
for any $h_1,h_2\in K_\psi$, we have
\begin{equation}
\jap{A_\theta (ph_1),ph_2}
=
\jap{A_\psi h_1,h_2}+c\jap{h_1,1-\psi}\jap{1-\psi,h_2}. 
\label{a11e}
\end{equation}
Moreover, 
\begin{equation}
\frac{\jap{A_\theta P_M(1-\theta),P_M(1-\theta)}}{\norm{P_M(1-\theta)}^2}
=
\frac{\jap{A_\psi(1-\psi),1-\psi}}{\norm{1-\psi}^2}
+
c\norm{1-\psi}^2, 
\label{a11ee}
\end{equation}
and
\begin{equation}
\frac1{4\pi}
\norm{P_M(1-\theta)}^2
=
\frac1{4\pi}
\norm{1-\psi}^2
+
\Im c\norm{1-\psi}^2. 
\label{a11f}
\end{equation}
\end{theorem}
\begin{proof}
By complex conjugation, \eqref{a11e} is equivalent to 
\begin{equation}
\jap{A_\theta^* (ph_2),ph_1}
=
\jap{A_\psi^* h_2,h_1}+\overline{c}\jap{h_2,1-\psi}\jap{1-\psi,h_1}. 
\label{a11g}
\end{equation}
Let us prove the last identity. By \eqref{a4}, we have
\begin{multline*}
A_\theta^*(ph_2)(x)-p(x)A_\psi^*h(x)
=
xp(x)h_2(x)-\Lambda_1(ph_2)-p(x)(xh_2(x)-\Lambda_1(h_2))
\\
=
-\Lambda_1(ph_2)+p(x)\Lambda_1(h_2)
=
(p(x)-p(\infty))\Lambda_1(h_2). 
\end{multline*}
Let us apply $P_M$ to both sides here and use Lemmas~\ref{lma.a3b} and \ref{lma.a1b}:
$$
P_MA_\theta^*(ph_2)
-
pA_\psi^*h_2
=
\varkappa p(1-\psi)\Lambda_1(h_2)
=
-\frac{\varkappa}{2\pi i}p(1-\psi)\jap{h_2,1-\psi}. 
$$
Denote 
$$
\overline{c}=-\frac{\varkappa}{2\pi i};
$$
taking the inner product with $ph_1$ and using the isometricity of $p$ on $K_\psi$, 
we arrive at \eqref{a11g}. 

In order to prove \eqref{a11ee}, it suffices to take $h_1=h_2=\overline{p(\infty)}(1-\psi)$ in \eqref{a11e}
and divide by $\norm{P_M(1-\theta)}^2$, using Lemma~\ref{lma.a3}.
Taking imaginary part and using \eqref{a5}, we arrive at 
\eqref{a11f}. 
\end{proof}

\section{Direct spectral problem: proof of Theorem~\ref{thm.z2}}\label{sec.aa}

Throughout this section, $u$ is a bounded rational symbol  with no poles in the closed upper half-plane, normalised so that $u(\infty)=0$ and $H_u$ is the Hankel operator \eqref{z2}. Furthermore, $\theta$ is the finite Blaschke product such that 
$$
\Ran H_u=K_\theta
$$
and normalised by $\theta(\infty)=1$. 

\subsection{Commutation relations for $H_u$ and $A_\theta$}

\begin{lemma}\label{lma.a5}
Let $u$, $\theta$ be as above. Then $H_u(1-\theta)=u$. Furthermore, 
\begin{align}
A_\theta^* H_u&=H_uA_\theta, 
\label{a12}
\\
A_\theta H_u^2-H_u^2A_\theta
&=\frac{i}{2\pi} \jap{\cdot, H_u u}(1-\theta)-\frac{i}{2\pi}\jap{\cdot,u}u. 
\label{a13}
\end{align}
\end{lemma}
\begin{proof}
Let $v_\zeta$ be the reproducing kernel of $H^2$, then 
\begin{equation}
H_uv_\zeta(x)=\frac1{2\pi i}\frac{u(x)-u(\zeta)}{x-\zeta}. 
\label{a2a}
\end{equation}
Comparing \eqref{a2a} with \eqref{a0d}, we find that 
$$
H_u v_\zeta=\frac1{2\pi i}(A_\theta^*-\zeta)^{-1}u.
$$
By \eqref{a14}, it follows that 
$$
\norm{(2\pi i \zeta)H_uv_\zeta+u}\to 0, \quad \abs{\zeta}\to\infty, 
$$
and so $u\in \overline{\Ran H_u}=K_\theta$.
Using this, we obtain 
$$
H_u(1-\theta)=\bbP_+(u-u\overline{\theta})=\bbP_+u=u. 
$$
Next, the kernel of $H_u$ is $\theta H^2$ and therefore $H_u(I-P_\theta)=0$. It follows that for any $f\in K_\theta$, 
$$
H_u S(\tau)f=H_u P_\theta S(\tau)f=H_u S_\theta(\tau) f. 
$$
Thus, restricting the commutation relation \eqref{z3} onto $K_\theta$ we obtain
$$
S_\theta(\tau)^* H_u=H_uS_\theta(\tau)
$$
or equivalently 
$$
e^{-i\tau A_\theta^*}H_u=H_ue^{i\tau A_\theta}.
$$
Differentiating this with respect to $\tau$
at $\tau=0$ 
and taking into account the anti-linearity of $H_u$, we arrive at
\eqref{a12}.

The proof of \eqref{a13} is a twice repeated application of \eqref{a12} and \eqref{a5}. 
We have
\begin{align*}
A_\theta H_u^2
&=
(A_\theta -A_\theta^*+A_\theta^*)H_u^2
=
\frac{i}{2\pi}\jap{\cdot,H_u^2(1-\theta)}(1-\theta)+A_\theta^*H_u^2
\\
&=
\frac{i}{2\pi}\jap{\cdot,H_u u}(1-\theta)+H_uA_\theta H_u. 
\end{align*}
Further, using the anti-linearity of $H_u$, 
\begin{align*}
H_uA_\theta H_u
&=
H_u(A_\theta -A_\theta^*+A_\theta^*)H_u
=
-\frac{i}{2\pi}\jap{1-\theta,H_u\cdot}H_u(1-\theta)+H_uA_\theta^*H_u
\\
&=
-\frac{i}{2\pi}\jap{\cdot,u}u+H_u^2A_\theta . 
\end{align*}
Combining these identities, we obtain \eqref{a13}. 
\end{proof}

\subsection{The action of $H_u$ on the cyclic subspace generated by $1-\theta$}

As in the introduction, we denote by $\lambda_1>\dots>\lambda_N>0$ the singular values of $H_u$ and by $E_{H_u}(\lambda_j)$ the corresponding Schmidt subspaces. We set for brevity $E_j:=E_{H_u}(\lambda_j)$ and denote by $P_j$ the orthogonal projection onto $E_j$. Observe that $H_u$ commutes with $H_u^2$ and therefore it commutes with $P_j$. We also set $g=1-\theta$ and 
$$
g_j:=P_j g, \quad u_j:=P_j u. 
$$
By the previous lemma, we have $H_ug=u$ and therefore $H_u g_j=u_j$. 

\begin{lemma}\label{lma.a6}
For any $j$, none of the elements $g$, $u$, $H_uu$ are orthogonal to $E_j$. 
Furthermore, we have
\begin{equation}
H_u u_j=\lambda_j e^{i\varphi_j}u_j, 
\qquad
H_u g_j=\lambda_j e^{-i\varphi_j}g_j
\label{a15}
\end{equation}
for some unimodular constants $e^{i\varphi_j}$. 
\end{lemma}
\begin{proof}
Since $H_u E_j=E_j$ (see Theorem~\ref{thm.a1}),
the three conditions $g\perp E_j$, $u\perp E_j$ and $H_uu \perp E_j$ are 
equivalent to each other. 

Suppose, to get a contradiction, that for some $j$ we have $P_j g=0$, 
i.e. the three orthogonality conditions above hold. 
For $f\in E_j$, applying the commutation relation 
\eqref{a13}, we get
$$
A_\theta H_u^2f=H_u^2A_\theta f,
$$
which can be rewritten as  
$$
H_u^2A_\theta f=\lambda A_\theta f.
$$
Thus, $A_\theta f\in E_j$, and so we obtain that $E_j$ is an invariant subspace for $A_\theta$. 
Further, by \eqref{a5}, we get $A_\theta^*f=A_\theta f$ for any $f\in E_j$. 
This contradicts  the complete non-selfadjointness of $A_\theta$
(see Lemma~\ref{lma.nsa}).

Let us prove that $H_u u_j$ and $u_j$ are collinear. 
Let $f\in E_j\cap u_j^\perp$; then 
$$
\jap{A_\theta H_u^2f,u_j}-\jap{H_u^2A_\theta f,u_j}
=
\lambda_j^2\jap{A_\theta f,u_j}-\lambda_j^2\jap{A_\theta f,u_j}=0. 
$$
By the commutation relation \eqref{a13}, this yields $f\perp H_uu$, hence $f\perp H_uu_j$. 
Thus, we get 
$$
E_j\cap u_j^\perp\subset E_j\cap (H_u u_j)^\perp. 
$$
Both of these subspaces are non-trivial and have codimension one in $E_j$. 
We conclude that these two subspaces must coincide, 
which means that $u_j$ and $H_u u_j$ 
are collinear.

Since 
$$
\norm{H_u u_j}^2=\jap{H_u^2u_j,u_j}=\lambda_j^2\norm{u_j}^2,
$$
from the previous step 
we get the first one of the relations \eqref{a15} with some unimodular constant $e^{i\varphi_j}$. 
Substituting $u_j =H_u g_j$, we obtain the second relation in \eqref{a15}. 
\end{proof}

\subsection{Proof of Theorem~\ref{thm.z2}}
The first two statements of the Theorem are already proved in Lemmas~\ref{lma.a5} and \ref{lma.a6}. It remains to prove the third one. 
Using \eqref{a12}, we obtain
\begin{equation}
\omega_j=\jap{A_\theta H_u g_j, H_u g_j}=\jap{g_j, H_uA_\theta H_u g_j}
=\jap{g_j, A_\theta^*H_u^2 g_j}=\lambda_j^2\jap{A_\theta g_j,g_j}. 
\label{b2a}
\end{equation}
Further, 
using \eqref{a5}, 
\begin{equation}
\Im \omega_j=\lambda_j^2\, \Im \jap{A_\theta g_j, g_j}=\frac{{\lambda_j}^2}{4\pi} \abs{\jap{g_j,g}}^2=\frac{\lambda_j^2\norm{g_j}^4}{4\pi}>0.
\label{b2b}
\end{equation}
\qed

\subsection{The matrix structure of $A_\theta$}

In preparation for our discussion of the inverse problem in the next section, here we discuss the matrix structure of $A_\theta$ with respect to the orthogonal decomposition \eqref{z9}. It turns out that the off-diagonal entries of the matrix of $A_\theta$ are rank one. We compute these entries here. Recall that $P_j$ be the orthogonal projection in $K_\theta$ onto $E_j$. 

\begin{lemma}
For all $j\not=k$ we have
\begin{equation}
P_k A_\theta P_j
=
\frac{i}{2\pi}
\frac{\lambda_j^2-\lambda_j\lambda_k e^{i(\varphi_j-\varphi_k)}}{\lambda_j^2-\lambda_k^2}
\jap{\cdot,g_j}g_k.
\label{b6}
\end{equation}
\end{lemma}
\begin{proof}
Let $f_j\in E_j$ and $f_k\in E_k$; then, taking the bilinear form of the 
second commutation relation in Lemma~\ref{lma.a5}, we obtain
$$
(\lambda_j^2-\lambda_k^2)\jap{A_\theta f_j,f_k}
=
\frac{i}{2\pi}\jap{f_j,H_u u_j}\jap{g_k,f_k}
-
\frac{i}{2\pi}\jap{f_j,u_j}\jap{u_k,f_k}. 
$$
Recall that 
$$
u_j=H_u g_j= \lambda_j e^{-i\varphi_j}g_j, \quad H_u u_j=H_u^2 g_j=\lambda_j^2 g_j.
$$
Substituting this into the above formula and dividing by $\lambda_j^2-\lambda_k^2$, 
we obtain
$$
\jap{A_\theta f_j,f_k}
=
\frac{i}{2\pi}
\frac{\lambda_j^2-\lambda_j\lambda_k e^{i(\varphi_j-\varphi_k)}}{\lambda_j^2-\lambda_k^2}
\jap{f_j,g_j}\jap{g_k,f_k},
$$
as required.
\end{proof}

We note here that the diagonal entries $P_jA_\theta P_j$ have more complicated structure, to be discussed in the next section.

\section{Inverse spectral problem: the spectral data and uniqueness}\label{sec.b}

\subsection{Preliminaries and notation}
The aim of this section is to prove the uniqueness part of Theorem~\ref{thm.z3} and to give an explicit formula for the symbol $u$ (and other objects) in terms of the spectral data \eqref{z8}. In order to motivate these formulas, we make some preliminary remarks and calculations. 
We follow the notation of the previous subsection; in particular, $u$, $\theta$, $\lambda_j$, $E_j$, $e^{i\varphi_j}$, $\omega_j$, $g_j$, $u_j$ are as above. We also denote for brevity 
$$
\nu_j=\norm{g_j}. 
$$
Then formula \eqref{b2b} becomes 
\begin{equation}
\Im \omega_j=\frac{\lambda_j^2\nu_j^4}{4\pi};
\label{b0}
\end{equation}
in particular, $\nu_j$ is determined by the spectral data \eqref{z8}.

First observe that by \eqref{a17}, we have
\begin{equation}
w(x)=\frac1{2\pi i}\jap{(A_\theta^*-x)^{-1}w,g}, 
\quad \Im x>0,
\label{eq.w}
\end{equation}
for any $w\in K_\theta$. Applying this to $u$ and recalling that by Lemma~\ref{lma.a6},
$$
u=\sum_{k=1}^N u_k=\sum_{k=1}^N \lambda_ke^{-i\varphi_k}g_k,
$$
we find
\begin{align}
u(x)&=\frac1{2\pi i}\sum_{k=1}^N \lambda_ke^{-i\varphi_k}\jap{(A_\theta^*-x)^{-1}g_k,g} 
\notag
\\
&=\frac1{2\pi i}\sum_{k,j=1}^N \lambda_ke^{-i\varphi_k}\jap{(A_\theta^*-x)^{-1}g_k,g_j} .
\label{b3aa}
\end{align}
Thus, $u(x)$ will be determined if we  compute all matrix entries
\begin{equation}
\jap{(A_\theta^*-x)^{-1}g_k,g_j}.
\label{b3a}
\end{equation}
This  leads us to the consideration of the matrix structure of $A_\theta^*$ in the orthogonal decomposition \eqref{z9}. We have computed the off-diagonal entries of $A_\theta^*$ in this decomposition in the previous section. Here we start by discussing the diagonal entries.  

In this section, we will use some matrix notation which we explain here. 
Below $\jap{\cdot,\cdot}$ is the inner product in $\bbC^N$. We denote by 
$\1_1,\dots,\1_N$ the standard basis in $\bbC^N$, 
and $\1=(1,\dots,1)^\top\in \bbC^N$.  
If $(\alpha_1,\dots,\alpha_N)$ are complex numbers, we will denote
by $D(\alpha)$ the diagonal $N\times N$ matrix with $\alpha_1,\dots,\alpha_N$ on the diagonal. 

\subsection{Diagonal elements of $A_\theta^*$}
Here we consider the operator $P_kA_\theta^* P_k$ acting in $E_k$. 
Recall that by Theorem~\ref{lma.a4}, we have
\begin{equation}
\jap{A_\theta (p_jh),p_jw}
=\jap{A_{\psi_j} h,w}+c_j\jap{h,1-\psi_j}\jap{1-\psi_j,w}, 
\quad  h,w\in K_{\psi_j},
\label{b6a}
\end{equation}
with some constants $c_j$. 
Further,  by Lemma~\ref{lma.a3}, we have
\begin{equation}
g_k =\overline{p_{k,\infty}}p_k(1-\psi_k), 
\label{b3}
\end{equation}
where $p_{k,\infty}=p_k(\infty)$.
\begin{lemma}
For $\Im x>0$, we have
\begin{equation}
R_k(x):=
\jap{(P_kA_\theta^* P_k-xI)^{-1}g_k,g_k}
=
\frac{\nu_k^2}{\norm{1-\psi_k}^2}
\frac{1-\psi_k(x)}{\frac1{2\pi i}+\overline{c_k}(1-\psi_k(x))}.
\label{b3b}
\end{equation}
\end{lemma}
\begin{proof}
Let us compute the vector
$$
f=(P_kA_\theta^* P_k-xI)^{-1}g_k
$$
by solving the equation
$$
P_kA_\theta^*f-xf=g_k, \quad f\in E_k.
$$
Using \eqref{b3} and writing $f=p_k h$, our equation becomes
$$
P_kA_\theta^*(p_k h)-xp_k h=\overline{p_{k,\infty}} p_k(1-\psi_k). 
$$
Let us take an inner product of this with an arbitrary element $p_kw\in E_k$, $w\in K_{\psi_k}$: 
$$
\jap{A_\theta^*(p_k h),p_kw}
-x\jap{h,w}=\overline{p_{k,\infty}}\jap{1-\psi_k,w}.
$$
Using \eqref{b6a}, we obtain
$$
\jap{(A_{\psi_k}^*-xI)h,w}
+
\overline{c_k}\jap{h,1-\psi_k}\jap{1-\psi_k,w}
=
\overline{p_{k,\infty}}\jap{1-\psi_k,w}.
$$
Since $w\in K_{\psi_k}$ is arbitrary, this implies 
$$
(A_{\psi_k}^*-x)h+\overline{c_k}\jap{h,1-\psi_k}(1-\psi_k)=\overline{p_{k,\infty}}(1-\psi_k). 
$$
Let us apply $(A_{\psi_k}^*-x)^{-1}$ and take the inner product with $1-\psi_k$: 
\begin{multline*}
\jap{h,1-\psi_k}
+
\overline{c_k}\jap{h,1-\psi_k}\jap{(A_{\psi_k}^*-x)^{-1}(1-\psi_k),1-\psi_k}
\\
=
\overline{p_{k,\infty}}\jap{(A_{\psi_k}^*-x)^{-1}(1-\psi_k),1-\psi_k}.
\end{multline*}
By \eqref{a17}, we have 
$$
1-\psi_k(x)=\frac1{2\pi i}\jap{(A_{\psi_k}^*-x)^{-1}(1-\psi_k),1-\psi_k},
$$
and therefore our equation becomes
$$
\jap{h,1-\psi_k}\bigl(\frac1{2\pi i}+\overline{c_k}(1-\psi_k(x)))
=
\overline{p_{k,\infty}}(1-\psi_k(x)).
$$
Finally, 
$$
R_k(x)=\jap{p_k h,g_k}
=
p_{k,\infty}\jap{h,1-\psi_k}
=
\abs{p_{k,\infty}}^2\frac{1-\psi_k(x)}{\frac1{2\pi i}+\overline{c_k}(1-\psi_k(x))}.
$$
Computing the norm in \eqref{b3}, we obtain
$$
\nu_k=\norm{g_k}=\abs{p_{k,\infty}}\norm{p_k(1-\psi_k)}=\abs{p_{k,\infty}}\norm{1-\psi_k}.
$$
Putting this together, we arrive at the required formula.
\end{proof}

\subsection{Matrix elements of the resolvent of $A_\theta^*$}\label{A}
Here we compute the matrix entries \eqref{b3a}. 
First let us introduce notation for the (normalised) matrix elements of $A_\theta$ with respect to the vectors $g_j$: 
$$
\calA_{kj}:=\frac1{\nu_j^2\nu_k^2}\jap{A_\theta g_j,g_k}.
$$
The off-diagonal entries have already appeared in the right hand side of \eqref{b6}: 
\begin{equation}\label{Aoff}
\calA_{kj}=\frac{i}{2\pi}
\frac{\lambda_j^2-\lambda_j\lambda_k e^{i(\varphi_j-\varphi_k)}}{\lambda_j^2-\lambda_k^2}, \quad j\not=k.
\end{equation}
The diagonal entries are given by \eqref{b2a}, viz.
\begin{equation}\label{Adiag}
\calA_{jj}=\frac{\omega_j}{\lambda_j^2\nu_j^4}=\frac{\omega_j}{4\pi \Im \omega_j}.
\end{equation}
Next, for any $x$ in the open upper half-plane we define an $N\times N$ matrix $Q(x)$ as follows:
\begin{align}
Q_{kj}(x)&=(\calA^*)_{kj}, \quad k\not=j,
\label{b26}
\\
Q_{kk}(x)&=1/R_k(x),
\label{b27}
\end{align}
where $R_k(x)$ is defined in the previous lemma. 
The following lemma is nothing but some linear algebra.

\begin{lemma}
For any $n,m$ we have 
\begin{equation}
\jap{(A_\theta^*-x)^{-1}g_m,g_n}
=
\jap{Q(x)^{-1}\1_m,\1_n}. 
\label{b29}
\end{equation}
\end{lemma}
\begin{proof}
Fix $m$ and denote 
$$
f=(A_\theta^*-x)^{-1}g_m.
$$
Our aim is to compute $\jap{f,g_n}$; 
the element $f$ satisfies the equation
\begin{equation}
(A_\theta^*-x)f=g_m.
\label{b30}
\end{equation}
Let us write $f=\sum_{j=1}^N f_j$ with $f_j\in E_j$. For every $j$ we have
$$
A_\theta^* f_j
=
\sum_{k\not=j}(\calA^*)_{kj}\jap{f_j,g_j}g_k
+
P_jA_\theta^*P_jf_j.
$$
Thus, our equation \eqref{b30} becomes a system
$$
(P_kA_\theta^*P_k-x)f_k+\sum_{j\not=k}(\calA^*)_{kj}\jap{f_j,g_j}g_k
=
\delta_{km}g_k, \quad k=1,\dots,N.
$$
Inverting $P_kA_\theta^*P_k-x$, we get
$$
f_k+\sum_{j\not=k}(\calA^*)_{kj}\jap{f_j,g_j}(P_kA_\theta^*P_k-x)^{-1}g_k
=
\delta_{km}(P_kA_\theta^*P_k-x)^{-1}g_k, 
\quad k=1,\dots,N.
$$
Let us take the inner product with $g_k$ and use the notation $R_k(x)$: 
$$
\jap{f_k,g_k}
+
\sum_{j\not=k}(\calA^*)_{kj}\jap{f,g_j}R_k(x)=\delta_{km}R_k(x), 
\quad k=1,\dots,N.
$$
Denote $\xi=(\xi_1,\dots,\xi_N)^\top\in \bbC^N$, $\xi_k=\jap{f_k,g_k}$; 
then we obtain
$$
(1/R_k(x))\xi_k+\sum_{j\not=k}(\calA^*)_{kj}\xi_j=\delta_{km}, 
\quad k=1,\dots,N.
$$
By the definition of $Q(x)$, this rewrites as
$$
Q(x)\xi=\1_m.
$$
Thus, $\xi=Q(x)^{-1}\1_m$ and 
$$
\jap{(A_\theta^*-x)^{-1}g_m,g_n}
=
\jap{f,g_n}
=
\xi_n=\jap{\xi,\1_n}
=
\jap{Q(x)^{-1}\1_m,\1_n},
$$
as required. 
\end{proof}

\subsection{Expression for $Q(x)$ in terms of the spectral data}
We have defined the diagonal entries of $Q(x)$ in terms of $R_k(x)$, and the expression \eqref{b3b} for $R_k(x)$ involves the constants $c_k$. It is not obvious that $Q(x)$ can be expressed entirely in terms of the spectral data \eqref{z8}. Let us show that this can be done. 

First we need some notation. 
For every $j$ and $\Im x>0$, we define 
\begin{equation}
b_j(x)=\frac1{\nu_j^{2}}\left(
i\frac{\norm{1-\psi_j}^2}{4\pi}\frac{1+\psi_j(x)}{1-\psi_j(x)}+\frac{\Re \jap{A_{\psi_j}(1-\psi_j),1-\psi_j}}{\norm{1-\psi_j}^2}-x\right).
\label{defb}
\end{equation}
Clearly, $b_j$ is determined by the spectral data; we recall that $\nu_j^2=\sqrt{4\pi \Im \omega_j}/\lambda_j$. By \eqref{b20aa} and Lemma~\ref{lma.b17}, this is a rational Herglotz function with the representation 
\begin{equation}
b_j(x)=\sum_k \frac{c_{jk}}{\alpha_{jk}-x}
\label{b20c}
\end{equation}
with some $c_{jk}\geq0$ and $\alpha_{jk}\in\bbR$. 
Again by  Lemma~\ref{lma.b17}, this function is Frostman-invariant (i.e. invariant with respect to the action of Frostman shifts on $\psi_j$). 

Let $Q(x)$ be as defined by \eqref{b26}, \eqref{b27}. Recall that $D(\alpha):=\diag\{\alpha_1,\dots,\alpha_N\}$. 
\begin{lemma}
The matrix $Q(x)$ can be expressed in terms of the spectral data by 
\begin{equation}
Q(x)=\calA^*-xD(\nu^2)^{-1}-D(b(x)), \quad \Im x>0.
\label{b28}
\end{equation}
\end{lemma}
\begin{proof}
For the off-diagonal entries, \eqref{b28} evidently agrees with \eqref{b26}. The issue is only to check that the diagonal entries agree. 
First we compute, using \eqref{b2a}, \eqref{b3} and \eqref{b6a}:
\begin{align*}
\frac{\Re \omega_k}{\lambda_k^2 \nu_k^2}
&-
\frac{\Re \jap{A_{\psi_k}(1-\psi_k),1-\psi_k}}{\norm{1-\psi_k}^2}
=
\frac{\Re \jap{A_\theta g_k,g_k}}{\norm{g_k}^2}
-
\frac{\Re \jap{A_{\psi_k}(1-\psi_k),1-\psi_k}}{\norm{1-\psi_k}^2}
\\
&=\frac{\Re \jap{A_\theta (p_k(1-\psi_k)),p_k(1-\psi_k)}}{\norm{p_k(1-\psi_k)}^2}
-
\frac{\Re \jap{A_{\psi_k}(1-\psi_k),1-\psi_k}}{\norm{1-\psi_k}^2}
\\
&=\frac{\Re \jap{A_\theta (p_k(1-\psi_k)),p_k(1-\psi_k)}}{\norm{1-\psi_k}^2}
-
\frac{\Re \jap{A_{\psi_k}(1-\psi_k),1-\psi_k}}{\norm{1-\psi_k}^2}
\\
&=\frac{\Re c_k\norm{1-\psi_k}^4}{\norm{1-\psi_k}^2}
=
 \norm{1-\psi_k}^2 \Re c_k.
\end{align*}
Writing the diagonal entry of the right hand side of \eqref{b28} and using the last formula, we get
\begin{align*}
\overline{\calA_{kk}}&-\frac{x}{\nu_k^2}-b_k(x)
=
\Re \calA_{kk}-i\Im \calA_{kk}-\frac{x}{\nu_k^2}-b_k(x)
\\
&=
\frac{\Re \omega_k}{\lambda_k^2\nu_k^4}-\frac{i}{4\pi}+\frac{1}{4\pi i}\frac{\norm{1-\psi_k}^2}{\nu_k^2}\frac{1+\psi_k(x)}{1-\psi_k(x)}-\frac{\Re\jap{A_{\psi_k}(1-\psi_k),1-\psi_k}}{\nu_k^2\norm{1-\psi_k}^2}
\\
&=
 \frac{\norm{1-\psi_k}^2}{\nu_k^2} \Re c_k
-\frac{i}{4\pi}
+\frac{1}{4\pi i}\frac{\norm{1-\psi_k}^2}{\nu_k^2}\frac{1+\psi_k(x)}{1-\psi_k(x)}.
\end{align*}
On the other hand, let us compute $1/R_k(x)$ and use \eqref{a11f}: 
\begin{align*}
1/R_k(x)&=\frac{\norm{1-\psi_k}^2}{\nu_k^2}\biggl(\frac1{2\pi i}\frac1{1-\psi_k(x)}+\overline{c_k}\biggr)
\\
&=\frac{\norm{1-\psi_k}^2}{\nu_k^2}\biggl(\Re c_k+\frac{1}{4\pi i}\frac{\nu_k^2}{\norm{1-\psi_k}^2}+\frac{i}{4\pi}+\frac{1}{2\pi i}\frac1{1-\psi_k(x)}\biggr)
\\
&=-\frac{i}{4\pi} +\frac{\norm{1-\psi_k}^2}{\nu_k^2}\biggl(\Re c_k+\frac1{4\pi i}\frac{1+\psi_k(x)}{1-\psi_k(x)}\biggr).
\end{align*}
Putting this together, we obtain the required identity. 
\end{proof}

\begin{remark}\label{simplespec}
Note that in the simple spectrum case, by performing a Frostman shift in each subspace $p_j K_{\psi_j}$, 
we can choose 
$$
\psi_j(x)=\frac{x-i}{x+i}
$$
for all $j$. Then a calculation shows that $\norm{1-\psi_j}^2=4\pi$ and the formula for $Q(x)$ becomes
$$
Q(x)=\calA^*-xD(\nu^2)^{-1}.
$$
\end{remark}

\subsection{Uniqueness and formula for $u$}\label{sec.b3}

Now we can put it all together and give a formula for $u$ in terms of the spectral data. The theorem below is more precise form of Theorem~\ref{thm.z3}(i).

\begin{theorem}\label{thm.uniq}
Let $u$ be a bounded rational function without poles in the closed upper half-plane, normalised so that $u(\infty)=0$. 
Then  $u$ is uniquely determined by the spectral data \eqref{z8} according to the following formula:
\begin{equation}
u(x)=\frac1{2\pi i}\jap{Q(x)^{-1}D(\lambda)D(e^{-i\varphi})\1,\1}, \quad \Im x>0.
\label{b25}
\end{equation}
Here
$$
Q(x)=\calA^*-xD(\nu^2)^{-1}-D(b(x)),
$$
where the matrix $\calA$ is defined by 
\begin{align*}
\calA_{kj}&=\frac{i}{2\pi}
\frac{\lambda_j^2-\lambda_j\lambda_k e^{i(\varphi_j-\varphi_k)}}{\lambda_j^2-\lambda_k^2}, \quad j\not=k,
\\
\calA_{jj}&=\frac{\omega_j}{\lambda_j^2\nu_j^4}=\frac{\omega_j}{4\pi \Im \omega_j},
\end{align*}
and $b_j$ are the functions defined by \eqref{defb}. 
\end{theorem}
\begin{proof}
Combining \eqref{b3aa} with \eqref{b29} gives
$$
u(x)=\frac1{2\pi i}\sum_{k,j=1}^N\lambda_ke^{-i\varphi_k}\jap{Q(x)^{-1}\1_k,\1_j}=\frac1{2\pi i}\jap{Q(x)^{-1}D(\lambda)D(e^{-i\varphi})\1,\1},
$$
as required. 
\end{proof}

\subsection{Formula for $g_j$}
For our proof of surjectivity in the following sections, we will need a formula for $g_j$, which appeared in the proof of the above theorem. 
It will be convenient to have it in vector form. 
\begin{lemma}
The vector 
$$
\bg(x)=(g_1(x),\dots,g_N(x))^\top
$$
can be expressed in terms of the spectral data by 
\begin{equation}
\bg(x)=\frac1{2\pi i} (Q(x)^\top)^{-1}\1.
\label{c8}
\end{equation}
\end{lemma}
\begin{proof}
As in the proof of Theorem~\ref{thm.uniq}, taking $w=g_k$ in \eqref{eq.w} and using \eqref{b29}, we obtain
$$
g_k(x)=\frac1{2\pi i}\jap{(A_\theta^*-x)^{-1}g_k,g}
=\frac1{2\pi i}\jap{Q(x)^{-1}\1_k,\1}
=\frac1{2\pi i}\jap{(Q(x)^{\top})^{-1}\1,\1_k}.
$$
Writing this in the vector form, we obtain \eqref{c8}.
\end{proof}

\subsection{Formula for $p_j$}\label{sec.b4}
For our proof of surjectivity in the following sections, it will be convenient to have a formula for the isometric multiplier $p_j$ in the representation $E_j=p_j K_{\psi_j}$. First let us fix the unimodular multiplicative constant  in the definition of $p_j$.  By \eqref{a16}, this can be done so  that 
\begin{equation}
H_u(p_j h)=\lambda_j p_j \psi_j \overline{h}, \quad h\in K_{\psi_j}.
\label{b1}
\end{equation}
This does not fix $p_j$ uniquely but up to a factor of $\pm1$ (observe that \eqref{a16} is invariant under the change $p\mapsto -p$). 

Take $h=1-\psi_j$ in \eqref{b1}:
$$
H_u(p_j(1-\psi_j))=-\lambda_j p_j(1-\psi_j). 
$$
On the other hand, \eqref{a15} gives the equation $H_ug_j=\lambda_je^{-i\varphi_j}g_j$ and \eqref{b3} gives a relation between $g_j$ and $p_j(1-\psi_j)$. Putting this together, after a little algebra we obtain 
\begin{equation}
e^{i\varphi_j}=-\frac{\overline{p_{j,\infty}}}{p_{j,\infty}}.
\label{b4}
\end{equation}
Next, by \eqref{b3} again we find 
$$
p_j=\frac1{\overline{p_{j,\infty}}}\frac{g_j}{1-\psi_j}.
$$
Taking into account \eqref{b4}, this becomes
\begin{equation}
p_j
=
\pm ie^{-i\varphi_j/2}\frac1{\abs{p_{j,\infty}}}\frac{g_j}{1-\psi_j}
=
\pm ie^{-i\varphi_j/2}\frac{{\norm{1-\psi_j}}}{\nu_j}\frac{g_j}{1-\psi_j},
\label{b8b}
\end{equation}
where the sign $\pm$ remains undetermined.

\subsection{Formula for $\theta(x)$} 

Below we give a formula for $\theta(x)$, $\Im x>0$, in terms of the spectral data. We do not need this formula in the rest of the paper, and so we give it without proof. Along with $Q(x)$, consider  the matrix
$$
Q^\#(x)=\calA-xD(\nu^2)^{-1}-D(b(x)), \quad \Im x>0;
$$
the difference with $Q(x)$ is that here we take $\calA$ instead of $\calA^*$. 
The matrix $Q^\#(x)$ is no longer necessarily invertible in $\bbC_+$. 
The inner function $\theta$ can be recovered from the spectral data by the formula 
$$
\theta(x)=\frac{\det Q^\#(x)}{\det Q(x)}, \quad \Im x>0. 
$$
The idea of the proof is to start from formula 
$$
1-\theta=\sum_{j=1}^N g_j=\sum_{j=1}^N \overline{p_{j,\infty}}p_j(1-\psi_j), 
$$
express $p_j$ according to \eqref{b8b} and rearrange the result using some matrix algebra similar to the one of Section~\ref{sec.f6} below.

\subsection{Comparison with \cite{OP2}}\label{sec.5.9}
For the reader's convenience, we compare the spectral data of Pocovnicu's paper \cite{OP2} (where the case of singular values of multiplicity one was considered) with the one of this paper. In \cite{OP2}, notation $\gamma_j$ is used for $\Re\omega_j$ and $\phi_j$ is used for $\varphi_j$. Notation $\lambda_j$ and $\nu_j$ in \cite{OP2} have the same meaning as here. The spectral data in \cite{OP2} is
$$
\left(
\{2\lambda_j^2\nu_j^2\}_{j=1}^N, 
\{4\pi \lambda_j^2\}_{j=1}^N, 
\{2\varphi_j\}_{j=1}^N,
\{\Re \omega_j\}_{j=1}^N
\right);
$$
these are the generalised action-angle variables for the Szeg\H{o} equation. Taking into account \eqref{b0}, we see that this spectral data is in a one-to-one correspondence with our spectral data \eqref{z8}.

\section{Inverse spectral problem: the surjectivity of the spectral map}\label{sec.c}

\subsection{The set-up}

Suppose we are given $N\in\bbN$, and the spectral data \eqref{z8}, 
where
\begin{itemize}
\item
$0<\lambda_1<\lambda_2<\dots<\lambda_N$ are real numbers;
\item
$\{\psi_j\}_{j=1}^N$ are finite Blaschke products with the normalisation condition
$\psi_j(\infty)=1$; 
\item
$\{e^{i\varphi_j}\}_{j=1}^N$ are unimodular complex numbers;
\item
$\{\omega_j\}_{j=1}^N$ are complex numbers with positive imaginary parts.
\end{itemize}
We define the numbers $\nu_j>0$ so that \eqref{b2b} holds, i.e. 
$$
4\pi \Im \omega_j=\lambda_j^2\nu_j^4. 
$$
With these parameters, let us define the $N\times N$ matrices $\calA$, $Q(x)$ 
as in Section~\ref{sec.b3}. 
Since $Q(x)=\calA^*-xD(\nu^2)^{-1}-D(b(x))$ and, by Lemma~\ref{lma.b17}, 
$$
b(x)=O(1/x), \quad \abs{x}\to\infty,
$$
we find that 
\begin{equation}
(Q(x))^{-1}=-\frac1x D(\nu^2)+O(1/\abs{x}^2), \quad \abs{x}\to\infty.
\label{c4}
\end{equation}

\begin{theorem}\label{thm.c1}
For every $x\in\bbC_+$, the matrix $Q(x)$ is invertible. Furthermore, 
$$
\sup_{x\in\bbC_+}\norm{Q(x)^{-1}}<\infty.
$$
\end{theorem}
Observe that as a consequence, the radial limits
$$
\lim_{\eps\to0+}Q(x+i\eps)^{-1}
$$
exist for a.e. $x\in\bbR$. 
We now define $u(x)$ by formula \eqref{b25}. It is evident that $u$ is a rational function; by \eqref{c4}, we have $u(x)\to0$ as $\abs{x}\to\infty$. Furthermore, by Theorem~\ref{thm.c1}, $u(x)$ has no poles in the closed upper half-plane. 
The main result of this section is 
\begin{theorem}\label{thm.c2}
The Hankel operator $H_u$ corresponds to the  spectral data \eqref{z8}
according to Theorem~\ref{thm.uniq}.
\end{theorem}

The main step of the proof is as follows. Let us define the functions $g_j(x)$ by \eqref{c8}. By definition, these are rational functions going to zero at infinity and by Theorem~\ref{thm.c1} they don't have poles in the closed upper half-plane. Thus, $g_j\in H^2$ for all $j$. 

\begin{theorem}\label{thm.c1a}
The eigenvalue equations 
$$
H_ug_j=\lambda_je^{-i\varphi_j}g_j, \quad j=1,\dots,N,
$$
or in vector form,
\begin{equation}
\bbP_+(u\overline{\bg})=D(\lambda)D(e^{-i\varphi})\bg,
\label{c8b}
\end{equation}
hold true.
\end{theorem}

In this section, we prove Theorems~\ref{thm.c1} and \ref{thm.c1a}. 
In the following section, we prove Theorem~\ref{thm.c2}.

\subsection{Algebraic properties of the matrix $\calA$}
For convenience of notation and also to make the connection with Hankel operators more transparent, 
let us define the anti-linear operator $\calH$ in $\bbC^N$ by 
\begin{equation}
\calH f=D(\lambda)D(e^{-i\varphi})\overline{f}, \quad \overline{f}=(\overline{f_1},\dots, \overline{f_N})^\top.
\label{c1a}
\end{equation}
\begin{lemma}\label{lma.c3}
The matrix  $\calA$ satisfies 
\begin{gather}
\Im \calA=\frac1{4\pi}\jap{\cdot, \1}\1,
\label{c2}
\\
\calA^*\calH=\calH \calA. 
\label{c3}
\end{gather}
\end{lemma}
\begin{proof}
First let us check \eqref{c2}.  For $j\not=k$ we have
$$
\frac1{2i}
(\calA_{kj}-\overline{\calA_{jk}})
=
\frac1{2i}\frac{i}{2\pi}\frac{1}{\lambda_j^2-\lambda_k^2}
\biggl(
\lambda_j^2-\lambda_j\lambda_ke^{i\varphi_j}e^{-i\varphi_k}-\lambda_k^2+\lambda_j\lambda_ke^{-i\varphi_k}e^{i\varphi_j}
\biggr)
=
\frac{1}{4\pi},
$$
which agrees with the right hand side of \eqref{c2}. 
For $j=k$ we have
$$
\Im \calA_{jj}=\frac{1}{4\pi}, 
$$
which again agrees with  \eqref{c2}. 
The second identity \eqref{c3} can be written as
$$
\overline{\calA}_{jk}\lambda_j e^{-i\varphi_j}=\overline{\calA}_{kj}\lambda_ke^{-i\varphi_k}
$$
in terms of the matrix entries. 
For $j\not=k$, the matrix $\calA$ satisfies this relation by an inspection of the definition of $\calA_{jk}$; 
for $j=k$ this relation is trivially true. 
\end{proof}

\begin{lemma}\label{lma.b4}
The eigenvalues of $\calA$ lie in the open upper half-plane. 
\end{lemma}
\begin{proof}
Since by \eqref{c2} we have $\Im \calA\geq0$, the question reduces to proving that $\calA$ has no real eigenvalues. 
Assume, to get a contradiction, that $\Ker (\calA-\lambda I)\not=\{0\}$ for some $\lambda\in \bbR$. 
Since $\Im \calA\geq0$, from here we easily check that 
$$
\Ker (\calA-\lambda I)=\Ker (\calA^*-\lambda I).
$$
Now let $f\in \Ker (\calA-\lambda I)$; by \eqref{c3} we have 
$$
\calA^*\calH f=\lambda\calH f, 
$$
i.e. $(\calA^*-\lambda I)\calH f=0$, and therefore $(\calA-\lambda I)\calH f=0$.
Thus, we see that $\Ker (\calA-\lambda I)$ is an invariant subspace of $\calH$. 
It follows that it is also an invariant subspace of the linear Hermitian operator $\calH^2=D(\lambda)^2$. 
It is also clear from \eqref{c2} that $\Ker (\calA-\lambda I)$ is orthogonal to the vector $\1$. 
But this vector is clearly cyclic for $D(\lambda)^2$. This contradiction completes the proof. 
\end{proof}

\subsection{Proof of Theorem~\ref{thm.c1}}
Let $\sigma\subset\{1,\dots,N\}$; we shall denote by $P_\sigma$ 
the orthogonal projection from $\bbC^N$ onto the $\abs{\sigma}$-dimensional subspace
$$
\{x\in\bbC^N: x_s=0\quad \text{ for } s\notin\sigma\}.
$$
We shall denote by $\overline{\bbC_+}$ the closed upper half-plane $\Im z\geq0$. 
\begin{lemma}\label{lma.b5}
Let $\calA$ be an $N\times N$ matrix such that
for any subset $L\subset\{1,\dots,N\}$, and for all $\beta_\ell\in \overline{\bbC_+}$, 
$\ell\in L$, 
the $\abs{L}\times\abs{L}$ matrix 
$$
D(\beta)+P_L\calA P_L^*
$$
is invertible.
Then 
$$
\sup_{\alpha_1,\dots,\alpha_N\in\overline{\bbC_+}}\norm{(D(\alpha)+\calA)^{-1}}<\infty.
$$
\end{lemma}
\begin{proof}
Assume, to get a contradiction, that there exist sequences
$\alpha^{(n)}\in (\overline{\bbC_+})^N$ and $X^{(n)}\in \bbC^N$ such that 
$\norm{X^{(n)}}=1$ for all $n$ and 
$$
\norm{(D(\alpha^{(n)})+\calA)X^{(n)}}\to0, \quad n\to\infty.
$$
After extracting a subsequence, we can achieve
$$
X^{(n)}\to X, \quad \norm{X}=1.
$$
Furthermore, again extracting subsequences, we can split the index set $\{1,\dots,N\}$
into disjoint subsets $J\cup L$ as follows: 
\begin{align*}
J&=\{j\in\{1,\dots,N\}: \abs{\alpha_j^{(n)}}\to\infty, \quad n\to\infty\}; 
\\
L&=\{\ell\in\{1,\dots,N\}: \alpha_\ell^{(n)}\to\alpha_\ell\in\overline{\bbC_+}\}.
\end{align*}
Now for every $j\in J$ we have
$$
\alpha_j^{(n)}X_j^{(n)}+(\calA X^{(n)})_j\to0,
$$
as $n\to\infty$, and therefore
$$
\alpha_j^{(n)}X_j^{(n)}\to -(\calA X)_j. 
$$
It follows that $X_j^{(n)}\to0$ as $n\to\infty$, and so $X_j=0$. 

Next, for every $\ell\in L$ we have
$$
\alpha_\ell^{(n)}X_\ell^{(n)}+(\calA X^{(n)})_\ell\to0,
$$
which yields
$$
\alpha_\ell X_\ell +(\calA X)_\ell=0. 
$$
Denoting $\beta_\ell=\alpha_\ell$, $\ell\in L$, 
this can be written as
$$
D(\beta)P_L X+P_L \calA X=0. 
$$
But by the previous step, $X=P_L^*P_LX$ and so we obtain
$$
D(\beta)P_L X+P_L \calA P_L^*P_LX=0. 
$$
By the assumption of the invertibility, we conclude $P_LX=0$, and so $X=0$ --- contradiction!
\end{proof}

\begin{proof}[Proof of Theorem~\ref{thm.c1}]
It suffices to prove the corresponding statement for $Q(x)^*$ in place of $Q(x)$. 
Further, since
$$
Q(x)^*=\calA-\overline{x}D(\nu^2)^{-1}-D(b(x))^*,
$$
and $\Im (x\nu_j^{-2}+b_j(x))>0$ for all $x\in\bbC_+$,
it suffices to prove that for any $\zeta_1,\dots,\zeta_N\in \bbC_+$ the matrix
$\calA+D(\zeta)$ is invertible and 
$$
\sup_{\zeta_1,\dots,\zeta_N\in \bbC_+}
\norm{(\calA+D(\zeta))^{-1}}<\infty.
$$
Let us show that this follows from the previous lemma. For $\zeta_j\in \bbC_+$, write 
$\zeta_j=\alpha_j+i\beta_j$ with $\alpha_j\in\bbR$, $\beta_j>0$. 
First, notice that Lemma~\ref{lma.b4} remains valid if we replace $\calA$ by 
$\calA+D(\alpha)$. Next, since 
$$
\Im (\calA+D(\zeta))=\Im \calA+D(\beta)\geq0,
$$
it is clear that all eigenvalues of $\calA+D(\zeta)$ lie in $\overline{\bbC_+}$. 
If $(\calA+D(\zeta))f=\lambda f$ for some $\lambda\in \bbR$, then taking the imaginary 
part of the quadratic form, we obtain $\Im \jap{\calA f,f}=0$ and so $\lambda$ is an eigenvalue
of $\calA+D(\alpha)$, which is impossible. Thus, 
all eigenvalues of $\calA+D(\zeta)$ lie in $\bbC_+$.

Finally, considering any square submatrix of $\calA$, we observe that it has the same
structure as $\calA$ itself, and so the above argument applies to this submatrix. 
It follows that the hypothesis of Lemma~\ref{lma.b5} is satisfied, and we arrive 
at the required result. 
\end{proof}

\subsection{Orthogonality of $g_j$}
Our aim here is to prove that $g_j$ form an orthogonal set in $H^2$, normalised by 
$$
\norm{g_j}=\nu_j. 
$$ 
In view of \eqref{c8}, this is a consequence of the following 
\begin{lemma}\label{lma.c6}
For any $X,Y\in \bbC^N$, we have 
\begin{equation}
\frac1{4\pi^2}\int_{-\infty}^\infty \jap{Q(x)^{-1}X,\1}\overline{\jap{Q(x)^{-1}Y,\1}}dx=\jap{D(\nu^2)X,Y}. 
\label{c5}
\end{equation}
\end{lemma}
\begin{proof}
As a first step, let us prove the identity 
\begin{equation}
\Im Q(x)^{-1}=\frac1{4\pi}Q^*(x)^{-1}\bigl(\jap{\cdot,\1}\1\bigr)Q(x)^{-1}, \text{ a.e. $x\in\bbR$.}
\label{c6}
\end{equation}
Recall that 
$$
Q(x)=\calA^*-xD(\nu^2)^{-1}-D(b(x)), 
$$
and $\Im b_j(x)=0$ for a.e. $x\in\bbR$. 
It follows that 
$$
\Im Q(x)=\Im \calA^*=-\Im \calA=-\frac1{4\pi}\jap{\cdot, \1}\1, \quad \text{ a.e. $x\in\bbR$.}
$$
From here we get the required identity \eqref{c6}. 

Next, observe that by \eqref{c6}, the integrand in the l.h.s. of \eqref{c5} rewrites as
\begin{align*}
\frac1{4\pi^2}
\jap{Q(x)^{-1}X,\1}\overline{\jap{Q(x)^{-1}Y,\1}}
&=
\frac1{\pi}
\jap{\Im Q(x)^{-1} X,Y}
\\
&=
\frac1{2\pi i}
\biggl(\jap{Q(x)^{-1}X,Y}-\overline{\jap{Q(x)^{-1}Y,X}}\biggr)
\end{align*}
for a.e. $x\in \bbR$. 
Therefore, the l.h.s. in \eqref{c5} rewrites as
$$
\frac1{2\pi i}\lim_{R\to\infty}\int_{-R}^R \jap{Q(x)^{-1}X,Y} dx
-
\frac1{2\pi i}\lim_{R\to\infty}\int_{-R}^R \overline{\jap{Q(x)^{-1}Y,X}}dx.
$$
Deforming the integration contour from $[-R,R]$ 
to the upper semi-circle of radius $R$ centered at the origin and using \eqref{c4}, we obtain  
$$
\frac1{2\pi i}\lim_{R\to\infty}\int_{-R}^R \jap{Q(x)^{-1}X,Y} dx=\frac12 \jap{D(\nu^2)X,Y}, 
$$
and, by complex conjugation
$$
\frac1{2\pi i}\lim_{R\to\infty}\int_{-R}^R \overline{\jap{Q(x)^{-1}Y,X}}dx=-\frac12 \jap{D(\nu^2)X,Y}. 
$$
Putting this together, we obtain \eqref{c5}. 
\end{proof}

\subsection{The action of $H_u$ on $g_j$}
Our aim here is to prove Theorem~\ref{thm.c1a}.
We recall that by our definitions, 
\begin{equation}
u(x)=\sum_{j=1}^N \lambda_j e^{-i\varphi_j}g_j(x)=\jap{D(\lambda)D(\nu)\bg(x),\1}. 
\label{c9}
\end{equation}
\begin{lemma}\label{lma.c7}
Let $f\in L^2(\bbR)$ be such that $xf\in L^2(\bbR)$. 
Then 
$$
\bbP_+(xf)=x\bbP_+(f)+\frac1{2\pi i}\int_{-\infty}^\infty f(x)dx. 
$$
\end{lemma}
\begin{proof} 
We have 
$$
f(z)=\frac1{2\pi i}\int_{-\infty}^\infty \frac{f(x)}{x-z}dx, 
$$
and therefore
$$
\bbP_+(xf)(z)-z\bbP_+(f)(z)
=
\frac1{2\pi i}\int_{-\infty}^\infty \frac{xf(x)}{x-z}dx
-
\frac1{2\pi i}\int_{-\infty}^\infty \frac{zf(x)}{x-z}dx
=
\frac1{2\pi i}\int_{-\infty}^\infty f(x)dx,
$$
as required.
\end{proof}

\begin{lemma}\label{lma.c8}
We have the identity
\begin{equation}
(\calA-xD(\nu^2)^{-1})\bbP_+(u\overline{\bg})
-
\bbP_+(uD(b)\overline{\bg})
=
\frac1{2\pi i}D(\lambda)D(e^{-i\varphi})\1
-
\frac{u(x)}{2\pi i}\1. 
\label{c11}
\end{equation}
\end{lemma}

\begin{proof}
By the definition of $\bg$ (cf. \eqref{c8}), we have
\begin{equation}
Q^\top(x)\bg(x)=\frac1{2\pi i}\1.
\label{c12}
\end{equation}
Let us take the complex conjugate of this equation, multiply by $u(x)$ and apply
$\bbP_+$: 
$$
\bbP_+(uQ^*\overline{\bg})=-\frac{1}{2\pi i}u\1.
$$
By the definition of $Q(x)$ and by using Lemma~\ref{lma.c7}, we rewrite the l.h.s. as
\begin{multline*}
\bbP_+(uQ^*\overline{\bg})
=
\calA\bbP_+(u\overline{\bg})
-
\bbP_+(xuD(\nu^2)^{-1}\overline{\bg})
-
\bbP_+(D(b)u\overline{\bg})
\\
=
(\calA-xD(\nu^2)^{-1})\bbP_+(u\overline{\bg})
-\frac1{2\pi i}D(\nu^2)^{-1}\int_{-\infty}^\infty u(x)\overline{\bg(x)}dx
-\bbP_+(D(b)u\overline{\bg}). 
\end{multline*}
Observe that by \eqref{c12}, we have $Q^\top(x)\bg(x)\in H^\infty$ and also 
by the definition of $\bg(x)$ we have $x\bg(x)\in H^\infty$; thus, all the expressions
above are well defined. 

By \eqref{c9} and the orthogonality of $g_j$, we get 
$$
\int_{-\infty}^\infty u(x)\overline{\bg(x)}dx
=
D(\nu^2)D(\lambda)D(e^{-i\varphi})\1.
$$
Putting this together, we obtain the required identity. 
\end{proof}

\begin{lemma}\label{lma.c9}
We have the idenity
\begin{equation}
Q^*(x)D(\lambda)D(e^{-i\varphi})\bg(x)
=
\frac1{2\pi i}D(\lambda)D(e^{-i\varphi})\1
-
\frac{u(x)}{2\pi i}\1, 
\label{c13}
\end{equation}
for a.e. $x\in\bbR$. 
\end{lemma}
\begin{proof}
Since $\Im \calA=\frac1{4\pi}\jap{\cdot,\1}\1$ and $D(b)$ is real on $\bbR$, we have
\begin{multline*}
Q^*(x)D(\lambda)D(e^{-i\varphi})\bg(x)
=
(\calA-\calA^*+\calA^*-xD(\nu^2)^{-1}-D(b))D(\lambda)D(e^{-i\varphi})\bg(x)
\\
=
\frac{2i}{4\pi}\jap{D(\lambda)D(e^{-i\varphi})\bg(x),\1}\1 
+
Q(x)D(\lambda)D(e^{-i\varphi})\bg(x)
\end{multline*}
for a.e. $x\in\bbR$. 
By  \eqref{c9}, 
$$
\frac{2i}{4\pi}\jap{D(\lambda)D(e^{-i\varphi})\bg(x),\1}\1
=
-\frac{u(x)}{2\pi i}\1. 
$$
Further, by the commutation relation \eqref{c3} we see that $Q(x)$ satisfies
$$
D(\lambda)D(e^{-i\varphi})Q(x)^\top=Q(x)D(\lambda)D(e^{-i\varphi}).
$$
Using this formula, we obtain
$$
Q(x)D(\lambda)D(e^{-i\varphi})\bg(x)
=
D(\lambda)D(e^{-i\varphi})Q^\top(x)\bg(x)
=
\frac1{2\pi i} D(\lambda)D(e^{-i\varphi})\1,
$$
where we have used \eqref{c12} at the last step. 
Putting this together, we obtain the required identity.
\end{proof}
The next lemma involves multiplication by $b_j(x)$ on the real axis. Here we need to proceed with caution because $b_j(x)$ has poles on the real axis, see \eqref{b20c}. 
We claim that 
\begin{equation}
D(b(x))Q(x)^{-1}\in H^\infty(\bbC_+) \quad \text{and}\quad
D(b(x))u(x)\in H^\infty(\bbC_+).
\label{c10}
\end{equation}
Indeed, from 
$$
Q(x)=\calA^*-xD(\nu^2)^{-1}-D(b(x))
$$
it is clear that $Q(x)$ has singularities at the same points as $D(b(x))$ and so these singularities cancel out. As an alternative argument, one can write 
$$
I=(\calA^*-xD(\nu^2)^{-1})Q(x)^{-1}-D(b(x))Q(x)^{-1}, \quad x\in\bbC_+;
$$
by Theorem~\ref{thm.c1}, one finds that the product $D(b(x))Q(x)^{-1}$ is 
bounded outside a neighbourhood  of infinity. 
On the other hand, this product is bounded in the neighbourhood of infinity
because both factors are bounded there. This gives the first inclusion in \eqref{c10}; the second one follows by recalling that the definition \eqref{b25} of $u(x)$ involves $Q(x)^{-1}$. 

\begin{lemma}\label{lma.c10}
We have the identity
$$
\bbP_+(Q^*(x)^{-1}\bbP_+(uD(b)\overline{\bg}))
=
\bbP_+(Q^*(x)^{-1}D(b)\bbP_+(u\overline{\bg})). 
$$
\end{lemma}
Note that by \eqref{c10},  both sides here are 
well defined. 
\begin{proof}
Let us take the inner product of the left side of the required identity with an 
arbitrary element $\bbf\in H^2$: 
\begin{multline*}
\jap{\bbP_+Q^*(x)^{-1}\bbP_+(uD(b)\overline{\bg}), \bbf}
=
\jap{\bbP_+(uD(b)\overline{\bg}), Q(x)^{-1}\bbf}
\\
=
\jap{uD(b)\overline{\bg}, Q(x)^{-1}\bbf}
=
\jap{u\overline{\bg}, D(b)Q(x)^{-1}\bbf}. 
\end{multline*}
Since $D(b)Q(x)^{-1}\in H^\infty$, we have
\begin{multline*}
\jap{u\overline{\bg}, D(b)Q(x)^{-1}\bbf}
=
\jap{\bbP_+(u\overline{\bg}), D(b)Q(x)^{-1}\bbf}
\\
=
\jap{Q^*(x)^{-1}D(b)\bbP_+(u\overline{\bg}), \bbf}
=
\jap{\bbP_+(Q^*(x)^{-1}D(b)\bbP_+(u\overline{\bg})), \bbf}.
\end{multline*}
This proves the required identity.
\end{proof}

\begin{proof}[Proof of Theorem~\ref{thm.c1a}]
Let us apply $Q^*(x)^{-1}$ to both sides of \eqref{c13} and then use \eqref{c11}: 
\begin{multline*}
D(\lambda)D(e^{-i\varphi})\bg(x)
=
Q^*(x)^{-1}\biggl(
\frac1{2\pi i}D(\lambda)D(e^{-i\varphi})\1
-
\frac{u(x)}{2\pi i}\1\biggr)
\\
=Q^*(x)^{-1}
(\calA-xD(\nu^2)^{-1})\bbP_+(u\overline{\bg})
-
Q^*(x)^{-1}\bbP_+(uD(b)\overline{\bg}).
\end{multline*}
Next, let us apply $\bbP_+$ to both sides and use Lemma~\ref{lma.c10}:
\begin{multline*}
D(\lambda)D(e^{-i\varphi})\bg(x)
=
\bbP_+Q^*(x)^{-1}
(\calA-xD(\nu^2)^{-1})\bbP_+(u\overline{\bg})
-
\bbP_+Q^*(x)^{-1}\bbP_+(uD(b)\overline{\bg})
\\
=
\bbP_+Q^*(x)^{-1}
(\calA-xD(\nu^2)^{-1})\bbP_+(u\overline{\bg})
-
\bbP_+Q^*(x)^{-1}D(b)\bbP_+(u\overline{\bg})
\\
=
\bbP_+Q^*(x)^{-1}Q(x)\bbP_+(u\overline{\bg})
=
\bbP_+(u\overline{\bg})
=
H_u \bg.
\end{multline*}
This is exactly the required eigenvalue equation in the vector form \eqref{c8b}.
\end{proof}

\section{Proof of Theorem~\ref{thm.c2}}\label{proof6.2}

In this section, we complete the proof of  the surjectivity of the spectral map. An important step consists in establishing that the rational function $p_j$ is indeed an isometric multiplier on the model space $K_{\psi_j}$. An ingredient of this proof is a representation of the functions $p_j$ in terms of a completely non-unitary contraction on $\bbC^N$, which is inspired by Sarason's work \cite{Sarason}.

\subsection{$p_j$ is an isometric multiplier on $K_{\psi_j}$}\label{sec.f6}
Let us define $p_j$ as in \eqref{b8b}, selecting for definiteness the sign ``$+$'': 
$$
p_j
=
ie^{-i\varphi_j/2}\frac{\norm{1-\psi_j}}{\nu_j} \frac{g_j}{1-\psi_j}.
$$
In vector form, denoting 
$$
\bp=(p_1,\dots,p_N)^\top, 
$$
recalling formula \eqref{c8} for $\bg$ and setting $\gamma_j=\frac1{2\sqrt{\pi}}\norm{1-\psi_j}$, we obtain 
$$
\bp=\frac1{\sqrt{\pi}}D(e^{-i\varphi/2})D(\gamma)D(\nu)^{-1}D(1-\psi)^{-1}(Q(x)^\top)^{-1}\1. 
$$
We need to rearrange this expression as follows.
\begin{lemma}
We have 
\begin{equation}
\bp=D(e^{-i\varphi/2})(I-BD(\psi))^{-1}\beta,
\label{b17}
\end{equation}
where $\beta\in\bbC^N$ and $B$ is a completely non-unitary contraction in $\bbC^N$ (with respect to the usual Euclidean norm), satisfying 
\begin{equation}
BB^*+\jap{\cdot,\beta}\beta=I.
\label{b17aa}
\end{equation}
\end{lemma}
\begin{proof}
First we rewrite formula \eqref{b28} for $Q(x)$ as 
$$
Q(x)=\calA_1^*-iD(\gamma^2)D(\nu^2)^{-1}D\bigl(\tfrac{1+\psi}{1-\psi}\bigr),
$$
where
$$
\calA_1:=\calA-\diag\biggl\{\frac{\Re\jap{A_{\psi_j}(1-\psi_j),1-\psi_j}}{\nu_j^2\norm{1-\psi_j}^2}\biggr\}. 
$$
Then we have
\begin{align*}
D(\gamma)&D(\nu)^{-1}D(1-\psi)^{-1}(Q(x)^\top)^{-1} 
\\
&=
D(\gamma)D(\nu)^{-1}D(1-\psi)^{-1}
\bigl(\overline\calA_1-iD(\gamma^2)D(\nu^2)^{-1}D(1+\psi)D(1-\psi)^{-1}\bigr)^{-1}
\\
&=\bigl(\overline\calA_1 D(\gamma)^{-1}D(\nu)D(1-\psi)-iD(\gamma)D(\nu)^{-1}D(1+\psi)\bigr)^{-1}
\\
&=\bigl(D(\nu)D(\gamma)^{-1}\overline\calA_1 D(\gamma)^{-1}D(\nu)D(1-\psi)-iD(1+\psi)\bigr)^{-1}D(\gamma)^{-1}D(\nu).
\end{align*}
Denote for brevity
$$
\calA_2=D(\nu)D(\gamma)^{-1}\calA_1 D(\gamma)^{-1}D(\nu),
$$
then 
\begin{align*}
\bigl(\overline{\calA_2} &D(1-\psi)-iD(1+\psi)\bigr)^{-1}D(\gamma)^{-1}D(\nu)
\\
&=
\bigl(\overline{\calA_2}-i-(\overline{\calA_2}+i)D(\psi)\bigr)^{-1}D(\gamma)^{-1}D(\nu)
\\
&=
(I-BD(\psi))^{-1}(\overline{\calA_2}-i)^{-1}D(\gamma)^{-1}D(\nu), 
\end{align*}
where
$$
B=(\overline{\calA_2}-i)^{-1}(\overline{\calA_2}+i).
$$
This yields \eqref{b17} with 
$$
\beta=\frac1{\sqrt{\pi}}(\overline{\calA_2}-i)^{-1}D(\gamma)^{-1}D(\nu)\1.
$$
Next, let us prove \eqref{b17aa}. We will see that this is a consequence of the rank one relation \eqref{c2} for $\calA$.
Since
$$
\Im \calA=\Im \calA_1=\frac1{4\pi}
\jap{\cdot,\1}\1, 
$$
we have
$$
\Im \calA_2=\frac1{4\pi}
\jap{\cdot,D(\gamma)^{-1}D(\nu)\1}D(\gamma)^{-1}D(\nu)\1, 
$$
and therefore
\begin{align*}
BB^*
&=
(\overline{\calA_2}-i)^{-1}(\overline{\calA_2}+i)(\overline{\calA_2}^*-i)(\overline{\calA_2}^*+i)^{-1}
\notag
\\
&=
(\overline{\calA_2}-i)^{-1}(\overline{\calA_2}\overline{\calA_2}^*+I+2\Im \overline{\calA_2})(\overline{\calA_2}^*+i)^{-1}
\notag
\\
&=
I-(\overline{\calA_2}-i)^{-1}(4\Im \overline{\calA_2})(\overline{\calA_2}^*+i)^{-1}
\notag
\\
&=
I-\jap{\cdot,\beta}\beta.
\end{align*}
Finally, let us check that $B$ is a completely non-unitary contraction in  $\bbC^N$.
The fact that $B$ is contraction is clear from \eqref{b17a}.
To check that $B$ is completely non-unitary, let us first consider the matrix $\calA_2$
and check that  it is completely non-self-adjoint. 
Indeed, suppose $\calA_2$ has a real eigenvalue $\lambda$ with an eigenvector $f$; then 
$$
(\calA_1-\lambda D(\gamma)^2D(\nu^2)^{-1})D(\gamma)^{-1}D(\nu)f=0.
$$
But this is impossible for $f\not=0$, because $\calA_1-\lambda D(\gamma)^2D(\nu^2)^{-1}$ is completely 
non-self-adjoint (see Lemma~\ref{lma.b4}). 
Since $\calA_2$ is completely non-self-adjoint, so is $\overline{\calA_2}$; 
it follows that $B$ is completely non-unitary. 
\end{proof}

The isometricity of $p_j$ is a consequence of  the following lemma. 
\begin{lemma}\label{lma.c11}
Let $B$ be a completely non-unitary contraction 
in $\bbC^N$ and 
$$
BB^*+\jap{\cdot,\beta}\beta=I
$$
with some vector $\beta\in \bbC^N$. 
Let $\psi_1,\dots,\psi_N$ be
inner functions in $\bbC_+$, and let the vector $\bp$
be defined by 
$$
\bp(z)=(I-BD(\psi(z)))^{-1}\beta, \quad z\in\bbC_+.
$$
Then each $p_j$ is 
an isometric multiplier on $K_{\psi_j}$.
\end{lemma}
\begin{proof}
\emph{Step 1:}
Let $\zeta_1,\dots,\zeta_N$ be complex numbers in the closed unit disk, 
$\abs{\zeta_j}\leq1$. As $B$ is a completely non-unitary contraction, 
so is $BD(\zeta)$, and therefore, by a compactness argument, 
the norms
$$
\norm{(I-BD(\zeta))^{-1}}
$$
are bounded uniformly for $\abs{\zeta_j}\leq1$, $j=1,\dots,N$. 
It follows that the inverse 
$$
(I-BD(\psi(x)))^{-1}, \quad x\in \bbC_+,
$$
is analytic and bounded in $\bbC_+$.

\emph{Step 2:}
Let $x\in\bbR$; denote for brevity $A=BD(\psi(x))$. 
Observe that we have $BB^*=AA^*$. Furthermore, 
\begin{align*}
\Abs{((I-A)^{-1}\beta)_j}^2
&=
[(1-A)^{-1}(\jap{\cdot,\beta}\beta)(I-A^*)^{-1}]_{jj}
\\
&=
[(I-A)^{-1}(I-AA^*)(I-A^*)^{-1}]_{jj},
\end{align*}
and 
$$
(I-A)^{-1}(I-AA^*)(I-A^*)^{-1}
=
I+(I-A)^{-1}A+A^*(I-A^*)^{-1}. 
$$
It follows that 
$$
\abs{p_j(x)}^2
=
1+
[(I-BD(\psi(x)))^{-1}B]_{jj}\psi_j(x)
+
[B^*(I-D(\psi(x))^*B^*)^{-1}]_{jj}\overline{\psi_j(x)}.
$$
Let us multiply this by $\abs{h(x)}^2$, where $h\in K_{\psi_j}$. 
We obtain, for $x\in \bbR$, 
\begin{align*}
\abs{p_j(x)}^2\abs{h(x)}^2=&\abs{h(x)}^2
\\
&+
[(I-BD(\psi(x)))^{-1}B]_{jj}(\psi_j(x)\overline{h(x)})h(x)
\\
&\qquad +
\overline{[B^\top(I-D(\psi(x))B^\top)^{-1}]_{jj}(\psi_j(x)\overline{h(x)})h(x)}.
\end{align*}
Observe that the second term in the r.h.s. is the boundary value of a function 
in $H^1(\bbC_+)$, while the third term is the complex conjugate of such boundary value. 
It follows that the integrals over $\bbR$ of both these terms vanish, and so integrating yields
the required isometricity of $p_j$. 
\end{proof}

\subsection{The action of $H_u$ on $p_jK_{\psi_j}$} 

In Theorem~\ref{thm.c1a} above, we have checked the eigenvalue equation
\begin{equation}
H_ug_j=\lambda_je^{-i\varphi_j}g_j.
\label{c19}
\end{equation}
Here our aim is to compute the action of $H_u$ on the whole subspace 
$p_jK_{\psi_j}$. 

\begin{lemma}\label{lma.c12}
For every $j=1,\dots,N$ and for every $h\in K_{\psi_j}$, we have
\begin{equation}
H_u(p_j h)=\lambda_j p_j \psi_j \overline{h}. 
\label{c20}
\end{equation}
\end{lemma}
\begin{proof}
Recall that $p_j$ is defined by the formula
$$
p_j
=
i\sqrt{4\pi} e^{-i\varphi_j/2}\frac{\gamma_j}{\nu_j} \frac{g_j}{1-\psi_j}
$$
and we have already checked that $p_j$ is an isometric multiplier on $K_{\psi_j}$. 
The desired equation \eqref{c20} can be written on the real line as 
$$
\bbP_+(u\overline{p_j}\overline{h})=\lambda_j p_j \psi_j \overline{h},
$$
i.e. we need to check that 
$$
F:=u\overline{p_j}\overline{h}-\lambda_j p_j \psi_j \overline{h}\in H^2(\bbC_-). 
$$
Let us check this inclusion. 
Observe that by construction, $F$ is a rational function without poles on $\bbR$ and $F(x)=O(1/x^2)$ as $\abs{x}\to\infty$. Thus, we only need to check that $F$ has no poles in the open lower half-plane.

First note that by the same logic the eigenvalue equation \eqref{c19} can be 
transformed into the condition 
$$
G:=u\overline{g_j}-\lambda_j e^{-i\varphi_j}g_j\in H^2(\bbC_-). 
$$
Observe that $G$ is a rational function without poles in the closed lower half-plane.

Next, recalling the definition of $p_j$ and using that $\abs{\psi_j}=1$ on the real line, we find 
$$
F=-i\sqrt{4\pi}\frac{\gamma_j}{\nu_j} e^{i\varphi_j/2}
\biggl(\frac{u\overline{g_j}\overline{h}}{1-\overline{\psi_j}}
+
\lambda_j e^{-i\varphi_j}\frac{g_j \psi_j \overline{h}}{1-\psi_j}\biggr)
=
-i\sqrt{4\pi}\frac{\gamma_j}{\nu_j} e^{i\varphi_j/2}\frac{G\overline{h}}{1-\overline{\psi_j}}. 
$$
From this representation we see that $F$ has no poles in the open lower half-plane. The proof is complete. 
\end{proof}

\subsection{Identification of $\theta$}

\begin{lemma}\label{lma.c13}
Let $\theta$ be the finite Blaschke product such that $\Ran H_u=K_\theta$ and 
$\theta(\infty)=1$. Let $g$ be defined by
$$
g=\sum_{j=1}^N g_j,
$$
where $g_j$ is given by \eqref{c8}; 
then  $g=1-\theta$. 
\end{lemma}
\begin{proof}
By Theorem~\ref{thm.c1a}, we have 
$$
H_u g=\sum_{j=1}^N H_u g_j =\sum_{j=1}^N  \lambda_j e^{-i\varphi_j} g_j =u.
$$
On the other hand, we know from Lemma~\ref{lma.a5} that $H_u(1-\theta)=u$. 
It follows that $g-(1-\theta)\in \Ker H_u$. 

Further, by the definition of $g$ and by Theorem~\ref{thm.c1a}, we see that 
$g\in \Ran H_u$. Also, $1-\theta\in K_\theta=\Ran H_u$. 
It follows that $g-(1-\theta)\in \Ran H_u$; we conclude that $g-(1-\theta)=0$.
\end{proof}

\subsection{Identification of the range of $H_u$}

\begin{lemma}\label{lma.c14}
The range of $H_u$ is given by
$$
\Ran H_u=\bigoplus_{k=1}^N p_k K_{\psi_k}. 
$$
\end{lemma}
\begin{proof}
By Lemma~\ref{lma.c12}, the subspaces $p_k K_{\psi_k}$ are mutually orthogonal and 
$$
\bigoplus_{k=1}^N p_k K_{\psi_k}
\subset 
\Ran H_u. 
$$
It suffices to check that for some dense set $D$ in $\Ran H_u$, we have
$$
D\subset \bigoplus_{k=1}^N p_k K_{\psi_k}.
$$
We use Corollary~\ref{cr.a7} and Lemma~\ref{lma.c13}; let us prove that for any $\zeta\in\bbC_+$, 
$$
(A_\theta^*-\zeta)^{-1}g\in \bigoplus_{k=1}^N p_k K_{\psi_k}. 
$$
By the definition of $g$, it suffices to check that 
$$
(A_\theta^*-\zeta)^{-1}g_j\in \bigoplus_{k=1}^N p_k K_{\psi_k}
$$
for all $j=1,\dots,N$ and all $\zeta\in\bbC_+$.
From formula \eqref{a0d} for the resolvent of 
$A_\theta^*$ and from formula \eqref{c8} for $g_j$  we find
\begin{align*}
(A_\theta^*-\zeta)^{-1}g_j(x)
&=
\frac{g_j(x)-g_j (\zeta)}{x-\zeta}
=
\frac1{2\pi i}
\jap{Q(x)^{-1}(Q(\zeta )-Q(x))Q(\zeta)^{-1}\1_j, \1 }
\\
&=\frac1{2\pi i}
\jap{(Q(\zeta )-Q(x))Q(\zeta)^{-1}\1_j, \overline{(Q^\top(x))^{-1}\1} }\ .
\end{align*}
Recalling that 
$$
\bg(x)=\frac1{2\pi i} (Q^\top(x))^{-1}\1
$$
and on the other hand
\begin{align*}
Q(x)-Q(\zeta )
&=
-iD(\gamma^2)D(\nu^2)^{-1}
D\left ( \frac{1+\psi (\zeta)}{1-\psi (\zeta)}-\frac{1+\psi (x)}{1-\psi(x)} \right )
\\
&=
-iD(\gamma^2)D(\nu^2)^{-1}
D\left (\frac{-2(\psi (x)-\psi (\zeta))}{(1-\psi(\zeta))(1-\psi (x))}\right ),
\end{align*}
we obtain, for some constants $c_{jk}(\zeta )$, 
$$
\frac{g_j(x)-g_j (\zeta)}{x-\zeta}=\sum_{k=1}^N c_{jk}(\zeta )\frac{g_k(x)(\psi_k(x)-\psi_k(\zeta ))}{(1-\psi_k(x))(x-\zeta)}\ .
$$
Since $g_k(x)=\overline{p_{k,\infty}}p_k(x)(1-\psi_k(x))$,
we are left with a linear combination of terms of the form
$$
p_k(x)\frac{\psi_k(x)-\psi_k(\zeta )}{x-\zeta }
=
-p_k(x)(A_{\psi_k}^*-\zeta)^{-1}(1-\psi_k)\ ,
$$
which belong to $p_k K_{\psi_k}$. This completes the proof.
\end{proof}

\subsection{Identification of $\omega_j$} 

\begin{lemma}\label{lma.c15}
For any $X,Y\in\bbC^N$, we have
\begin{equation}
\int_{-\infty}^\infty \jap{Q(x)^{-1}D(b(x))X,\1}\overline{\jap{Q(x)^{-1}Y,\1}} dx=0.
\label{c21}
\end{equation}
\end{lemma}
\begin{proof}
Following the proof of Lemma~\ref{lma.c6}, we obtain 
\begin{multline*}
\frac1{4\pi^2}
\jap{Q(x)^{-1}D(b(x))X,\1}\overline{\jap{Q(x)^{-1}Y,\1}} 
\\
=\frac1{2\pi i}\bigl(
\jap{Q(x)^{-1}D(b(x))X,Y}-\overline{\jap{Q(x)^{-1}Y,D(b(x))X}}\bigr).
\end{multline*}
Since $b_j$ are real-valued on $\bbR$, we have 
$$
\overline{\jap{Q(x)^{-1}Y,D(b(x))X}}=\overline{\jap{D(b(x))Q(x)^{-1}Y,X}}
$$
for a.e. $x\in\bbR$. Therefore, the l.h.s. of \eqref{c21} rewrites as
$$
\frac1{2\pi i}
\lim_{R\to\infty}\int_{-R}^R \jap{Q(x)^{-1}D(b(x))X,Y}dx
-
\frac1{2\pi i}
\lim_{R\to\infty}\int_{-R}^R \overline{\jap{D(b(x))Q(x)^{-1}Y,X}} dx.
$$
Deforming the integration contour as in the proof of Lemma~\ref{lma.c6}
and using that $b(x)=O(1/x)$ at infinity, 
we find that both limits are equal to zero. 
\end{proof}

\begin{lemma}\label{lma.c16}
For the operator $A_\theta$ corresponding to $H_u$, we have 
$$
\jap{A_\theta g_j,g_k}=\nu_j^2\nu_k^2\calA_{kj}
$$
and, in particular, 
$$
\jap{A_\theta g_j, g_j}=\frac{\omega_j}{\lambda_j^2}. 
$$
\end{lemma}
\begin{proof}
We shall prove the equivalent identity 
$$
\jap{A_\theta^* g_j,g_k}=\nu_j^2\nu_k^2(\calA^*)_{kj}.
$$
By \eqref{a4}, we have
$$
A_\theta^* g_j(x)=xg_j(x)-\Lambda_1(g_j). 
$$
Recall that 
$$
g_j(x)=\frac1{2\pi i}\jap{Q(x)^{-1}\1_j,\1}. 
$$
By the asymptotic formula \eqref{c4}, we find
$$
\Lambda_1(g_j)=-\frac1{2\pi i}\nu_j^2=
-\frac1{2\pi i}\jap{D(\nu^2)\1_j,\1},
$$
and therefore
\begin{align*}
A_\theta^* g_j(x)
&=
\frac1{2\pi i}
\jap{(xQ(x)^{-1}+D(\nu^2))\1_j,\1}
=
\frac1{2\pi i}
\jap{Q(x)^{-1}(xI+Q(x)D(\nu^2))\1_j,\1}
\\
&=
\frac1{2\pi i}
\jap{Q(x)^{-1}(xD(\nu^2)^{-1}+Q(x))D(\nu^2)\1_j,\1}
\\
&=
\frac1{2\pi i}
\jap{Q(x)^{-1}\calA^*D(\nu^2)\1_j,\1}
-
\frac1{2\pi i}
\jap{Q(x)^{-1}D(b(x))D(\nu^2)\1_j,\1}.
\end{align*}
It follows that 
\begin{multline*}
\jap{A_\theta^*g_j,g_k}
=
\frac1{4\pi^2} \int_{-\infty}^\infty \jap{Q(x)^{-1}\calA^*D(\nu^2)\1_j,\1}\overline{\jap{Q(x)^{-1}\1_k,\1}}dx
\\
-
\frac1{4\pi^2} \int_{-\infty}^\infty \jap{Q(x)^{-1}D(b(x))D(\nu^2)\1_j,\1}\overline{\jap{Q(x)^{-1}\1_k,\1}}dx.
\end{multline*}
Here the first term in the r.h.s equals 
$$
\jap{D(\nu^2)\calA^*D(\nu^2)\1_j,\1_k}=\nu_j^2\nu_k^2(\calA^*)_{kj}
$$
by Lemma~\ref{lma.c6} and the second one equals zero by Lemma~\ref{lma.c15}. 
\end{proof}

\subsection{Proof of Theorem~\ref{thm.c2}}
The theorem follows by putting together Theorem~\ref{thm.c1a} and the lemmas of this section. Indeed, by Theorem~\ref{thm.c1a} and by Lemma~\ref{lma.c14}, the set of singular values of $H_u$ is exactly $\{\lambda_n\}_{n=1}^N$. Again by Theorem~\ref{thm.c1a} and by Lemma~\ref{lma.c12}, the inner function $\psi_j$ and the unimodular constant $e^{i\varphi_j}$ corresponds to the eigenvalue $\lambda_j$. Finally, by Lemma~\ref{lma.c16}, the parameters $\omega_j$ correspond to $A_\theta$ and $g_j$. 
\qed

\section{The  Szeg\H{o} dynamics}\label{evol}

\subsection{Formulas for the Szeg\H{o} dynamics}

In this section we express the Szeg\H{o} dynamics for rational solutions in terms of the spectral data. 
The main result here is

\begin{theorem}\label{thm1}
Let $u$ be a solution of the cubic Szeg\H{o} equation
\begin{equation}
i\frac{d}{dt} u=\bbP_+(\abs{u}^2 u)
\label{sz}
\end{equation}
with $u|_{t=0}$ rational.
Then the solution is rational for all $t>0$, and the spectral data of $u$  satisfy 
the following law:
\begin{align}
\frac{d}{dt}\lambda_j&=\frac{d}{dt}\psi_j=0;
\label{s1}
\\
\frac{d}{dt}\varphi_j&=\lambda_j^2, 
\quad 
\label{s2}
\\
\frac{d}{dt}\omega_j&=\frac1{2\pi}\lambda_j^4\nu_j^4.
\label{s3}
\end{align}
\end{theorem}

Before coming to the proof of Theorem \ref{thm1}, observe the following. Since we know from \cite{OP2} that the initial value problem for the cubic Szeg\H{o} equation is wellposed in every Sobolev space $W^{s,2}(\bbC_+)$ for every $s\geq 1/2$, 
the uniqueness implies that it is enough to prove that the rational function corresponding to the spectral data defined by the evolution laws \eqref{s1}, \eqref{s2}, \eqref{s3} is indeed a solution of the cubic Szeg\H{o} equation.

We consider the spectral data evolving according to the evolution laws \eqref{s1}, \eqref{s2}, \eqref{s3}. For each $t>0$, we define the matrices $\calA$ and $Q(x)$ as in Section~\ref{sec.b}, suppressing the dependance on $t$ in our notation. 
We define the anti-linear operator $\calH$ in $\bbC^N$ as in \eqref{c1a}.
According to Section~\ref{sec.b}, the functions $u$ and $u_j$ can be recovered from the spectral data by the formulas 
\begin{align}
u_j(x)&=\frac1{2\pi i}\jap{Q(x)^{-1}\calH\1_j,\1}, 
\notag
\\
u(x)&=\frac1{2\pi i}\jap{Q(x)^{-1}\calH\1,\1}.
\label{eq.u}
\end{align}
Our aim is to differentiate \eqref{eq.u} with respect to $t$ and check that $u$ satisfies \eqref{sz}.

\subsection{The time derivative of $Q$}

We first observe that the equation \eqref{s3} 
means, in particular, that $\Im \omega_j$ are fixed by the dynamics. By \eqref{b0}, it follows that $\nu_j$ are also fixed. 

The time derivative of $\calH\1$ is straightforward to compute:
\begin{equation}
\frac{d}{dt}\calH\1=
D(\lambda)\frac{d}{dt}D(e^{-i\varphi})\1=
-i\calH^2 \calH\1=-i\calH^3\1.
\label{b9}
\end{equation}
Let us compute the time derivative of the matrix $\calA$. 
For $j\not=k$, the only time-dependant quantities of $\calA_{kj}$ are $\varphi_j$ and $\varphi_k$, and so we get
$$
\frac{d}{dt}\calA_{kj}
=
\frac{i}{2\pi}\frac{-\lambda_j\lambda_k(i\lambda_j^2-i\lambda_k^2)e^{i\varphi_j}e^{-i\varphi_k}}{\lambda_j^2-\lambda_k^2}
=
\frac{1}{2\pi}\lambda_j\lambda_ke^{i\varphi_j}e^{-i\varphi_k}.
$$
For the diagonal entries $\calA_{jj}$ we have, by our definitions, 
$$
\frac{d}{dt}\calA_{jj}=\frac1{\lambda_j^2\nu_j^4}\frac{d}{dt}\omega_j=\frac1{2\pi} \frac{\lambda_j^4\nu_j^4}{\lambda_j^2\nu_j^4}=\frac1{2\pi}\lambda_j^2. 
$$
Putting this together, we find
$$
\frac{d}{dt}\calA
=
\frac1{2\pi} \jap{\cdot,\calH\1}\calH\1.
$$
Since $b_j$ is independent of $t$, we obtain 
$$
\frac{d}{dt} Q(x)=\frac{1}{2\pi}\jap{\cdot,\calH\1}\calH\1, \quad \Im x>0,
$$
and finally, taking inverses, 
\begin{equation}
\frac{d}{dt} Q(x)^{-1}=-\frac{1}{2\pi}\jap{\cdot,(Q^*(x))^{-1}\calH\1}Q(x)^{-1}\calH\1, 
\label{b8}
\end{equation}
for $\Im x>0$ and, by taking limits, everywhere on the real axis apart from finitely many points.

\subsection{Formulas for $H_uu$ and $H_u^2 u$. }

One easily verifies the identity 
$$
uH_uu+H_u^2 u=\bbP_+(\abs{u}^2u). 
$$
Using this, one can rewrite the Szeg\H{o} equation in the following equivalent form:
\begin{equation}
i\frac{d}{dt} u=uH_uu+H_u^2 u.
\label{b00}
\end{equation}
Now, in preparation for what comes next, let us express $H_uu$ and $H_u^2 u$ 
in terms of the spectral data. We have:
\begin{align*}
H_uu&=\sum_j \lambda_j e^{i\varphi_j}u_j
=\frac1{2\pi i}\sum_j \lambda_j e^{i\varphi_j}\jap{Q(x)^{-1}\calH\1_j,\1}
=\frac1{2\pi i}\jap{Q(x)^{-1}\calH^2\1,\1},
\\
H_u^2u&=\sum_j \lambda_j^2 u_j
=\frac1{2\pi i}\sum_j \lambda_j^2 \jap{Q(x)^{-1}\calH\1_j,\1}
=\frac1{2\pi i}\jap{Q(x)^{-1}\calH^3\1,\1}.
\end{align*}

\subsection{Concluding the proof of Theorem~\ref{thm1}}

First we need an identity relating $Q$ and $\calH$. 
From the matrix identity \eqref{c3} we get
$$
Q(x)\calH =\calH Q^*(x), \quad \Im x>0.
$$
Passing to the inverses,
\begin{equation}
\calH(Q^*(x))^{-1}=Q(x)^{-1}\calH, 
\label{HQ}
\end{equation}
for $\Im x>0$ and, by taking limits, also everywhere on the real axis apart from finitely many points (the poles of $b_j$).

Using \eqref{b9} and \eqref{b8}, we find (suppressing the dependance of $x$)
\begin{align*}
i\frac{d}{dt}u&=\frac{i}{2\pi i}\frac{d}{dt}\jap{Q^{-1}\calH\1,\1}
\\
&=\frac{1}{2\pi}\jap{\bigl(\tfrac{d}{dt}Q^{-1}\bigr)\calH\1,\1}
+
\frac{1}{2\pi}\jap{Q^{-1}\bigl(\tfrac{d}{dt}\calH\1\bigr),\1}
\\
&=
\left(\frac1{2\pi i}\right)^2
\jap{\calH\1,(Q^*)^{-1}\calH\1}\jap{Q^{-1}\calH\1,\1}
+
\frac1{2\pi i} \jap{Q^{-1}\calH^3\1,\1}
\\
&=u(x)\frac1{2\pi i}
\jap{Q^{-1}\calH\1,\calH\1}
+H_u^2u.
\end{align*}
Using \eqref{HQ}, we transform the inner product in the right hand side as
$$
\jap{Q^{-1}\calH\1,\calH\1}
=
\jap{\calH (Q^*)^{-1}\1,\calH\1}
=
\jap{\calH^2\1,(Q^*)^{-1}\1}
=
\jap{Q^{-1}\calH^2\1,\1}.
$$
Putting this together, we obtain the Szeg\H{o} equation in the form \eqref{b00}.

\section{The genericity of turbulent solutions: proof of Theorem~\ref{genericgrowth}}\label{growth}
In this last section, we prove Theorem~\ref{genericgrowth}.
The key argument  is to establish that, if $u$ is a rational solution of the cubic Szeg\H{o} equation \eqref{z1} such that one of the singular values of the Hankel operator $H_u$ is multiple while the other singular values are simple, then the $L^2$ norm of $\partial_x u$ tends to infinity as $t$ tends to infinity. This can be achieved thanks to the representation of rational solutions obtained in previous sections.
\subsection{Upper bound for general rational solutions}
We start with a general a priori bound for rational solutions.
\begin{proposition}\label{upperbound}
If $u$ is a rational solution of the cubic Szeg\H{o} equation, then
$$\limsup_{t\to +\infty} \frac{\norm {\partial_xu(\cdot,t)}_{L^2}}{t}<+\infty  .$$
\end{proposition}
Before proceeding with the proof, we need a simple lemma. 
For every spectral data 
$$
(\lambda, \psi ,e^{i\varphi} , \omega ):=\biggl(\{\lambda_j\}_{j=1}^N, \{\psi_j\}_{j=1}^N ,\{e^{i\varphi_j}\}_{j=1}^N, \{\omega _j\}_{j=1}^N\biggr) \ , 
$$ 
denote by $\mathcal A (\lambda, \psi ,\varphi , \omega)$ the matrix defined in Subsection~\ref{A}.
We also denote by $\bbP_j$ the projector matrix onto the $j$'th direction in $\bbC^N$. 
\begin{lemma} \label{linearalgebra}
Fix $(\lambda,\psi, \omega )$; there exists $C>0$ such that, for every $j=1,\dots ,N$, we have
$$\forall \xi\in \bbR ^N\ ,\ \sup_{\varphi \in \bbT^N} \norm {\bbP_{j}(\calA(\lambda, \psi ,\varphi , \omega)^*+D(\xi))^{-1}}\leq \frac{C}{1+|\xi_j|}\ .$$
A similar result holds for $\calA (\lambda, \psi ,\varphi  , \omega)$ in place of $\calA(\lambda, \psi  ,\varphi  , \omega )^*$.
\end{lemma}
\begin{proof}
Let $X\in \bbC^N$ of norm $1$, and let 
$$Z:=(\calA(\lambda, \psi ,\varphi , \omega)^*+D(\xi))^{-1}X\ .$$
We want to prove that the components of $Z$ satisfy
$$|Z_j|\leq \frac{C}{1+|\xi_j|}\ .$$
We already know, from Lemma \ref{lma.b5} and from the proof of Theorem \ref{thm.c1}, that $\norm{Z}$ is bounded. Furthermore, one can easily check from the proof of Lemma \ref{lma.b5} that this estimate is uniform in $\varphi \in \bbT^N$. Then we come back to the equations in $Z$, which read
$$\xi_jZ_j+(\calA(\lambda, \psi ,\varphi , \omega)^*Z)_j=X_j$$
and this immediately leads to the required estimate. The proof for $\calA$ is similar.
\end{proof}
\begin{proof}[Proof of Proposition~\ref{upperbound}]
By Theorem \ref{thm.uniq} and Theorem \ref{thm1}, the rational solution of the cubic Szeg\H{o} equation reads
$$u(x,t)=\frac{1}{2\pi i}\jap{R_t(\xi(x,t))^{-1}X(t),\1 }\ $$
with, for every $\xi =(\xi_1,\dots,\xi_N)\in \overline{\bbC_-}^N$, 
$$R_t(\xi):=\mathcal A(\lambda, \psi ,\varphi (t), \omega (0))^*+D(\xi)\ , \ X(t):=D(\lambda )D(e^{-i\varphi (t)} )\1\ ,$$
and 
\begin{equation}\label{xij}
\xi_j(x,t):=\frac{\lambda_j^2}{2\pi}t-b_j(x)-\frac{x}{\nu_j^2}.
\end{equation}
Notice that the time dependence of $\xi_j(x,t)$ is due to the time dependence of $\Re \omega_j(t)$ coming from \eqref{s3}, and that 
$$
\partial_x\xi_j(x,t)=-b'_j(x)-\tfrac1{\nu_j^2}
$$
is independent of $t$. Also recall that $b_j$ is a rational Herglotz function, representable as in \eqref{b20c}. 
Consequently, denoting by $\bbP_j$ the projector matrix onto the $j$'th direction in $\bbC^N$, we have
\begin{equation} \partial_xu(x,t)=\sum_{j=1}^N\frac{i\partial_x\xi_j(x,t)}{2\pi }\jap{R_t(\xi (x,t))^{-1}\bbP_j R_t(\xi (x,t))^{-1}X(t) ,\1}\ .\label{du}\end{equation}
Write 
$\calA (t):=\calA(\lambda, \psi ,\varphi (t), \omega (0))$ for brevity and observe that 
\begin{align*}
\jap{R_t(\xi (x,t))^{-1}&\bbP_j R_t(\xi (x,t))^{-1}X(t),\1}
\\
&=\jap{\bbP_{j}(\calA(t)^*+D(\xi(x,t)))^{-1}X(t),\bbP_{j}(\calA(t)+D(\xi(x,t)))^{-1}\1}
\end{align*}
so that Lemma ~\ref{linearalgebra} and identity \eqref{du} lead to
$$\vert \partial_xu(x,t)\vert \lesssim \sum_{j=1}^N\frac{|\partial_x\xi_j(x,t)|}{1+\xi_j(x,t)^2}\ ,\ x\in \bbR\ ,$$
where $\lesssim$ denotes inequality up to a multiplicative constant. Let us fix $j$; 
observe that $b_j(x)+\frac{x}{\nu_j^2}$ is strictly increasing between the poles of $b_j$. We decompose the integral
$$\int_{\bbR}\frac{|\partial_x\xi_j(x,t)|^2}{(1+\xi_j(x,t)^2)^2}\, dx$$
into a finite sum of integrals over the open intervals between the adjacent poles of $b_j$ (plus two semi-infinite intervals).
Then on each interval the map $x\mapsto \xi_j(x,t)$ is strictly decreasing.  We write in each of these integrals
$$|\partial_x\xi_j(t,x)|\, dx=d\xi_j\ .$$
We have 
$$
\abs{b'_j(x)}\lesssim \bigl(b_j(x)+\tfrac{x}{\nu_j^2}\bigr)^2+1; 
$$
indeed, this follows by observing that both sides are rational functions with poles of second order located at the same points and by inspecting the behaviour at infinity. Next, we have, as $t\to+\infty$, 
$$
\abs{\partial_x\xi_j(t,x)}=\abs{b_j'(x)+\tfrac1{\nu_j^2}}
\lesssim
\bigl(b_j(x)+\tfrac{x}{\nu_j^2}\bigr)^2+1
=
\bigl(\xi_j(t,x)-\tfrac{\lambda_j^2}{2\pi}t\bigr)^2+1
\lesssim
\xi_j(t,x)^2+t^2.
$$
Plugging this estimate into each of our integrals and summing over $j$, we get the required bound
$$\norm {\partial_xu(\cdot,t)}_{L^2}^2\lesssim t^2\ .$$
\end{proof}

\subsection{Lower bound in the case of one multiple eigenvalue}
\begin{proposition}\label{lowerbound}
Let $u$ be a rational solution of the cubic Szeg\H{o} equation on the line such that $H_u$ has singular values $\lambda_1,\cdots, \lambda_N$, with $\lambda_1$ being multiple and $\lambda_j$ being simple for every $j\ge 2$.
Then
$$\liminf_{t\to +\infty}\frac{\norm{\partial_xu(t)}_{L^2}}{t}>0\ .$$
\end{proposition}
\begin{proof}
We decompose $\partial_xu(x,t)$ as in the proof of Proposition~\ref{upperbound}, starting from \eqref{du}. Since $\lambda_j$ is simple for $j\geq 2$,
it is easy to check (see e.g. Remark ~\ref{simplespec}) that $\xi_j(x,t)$ is linear in $x$ and therefore that $\partial_x\xi_j(x,t)$ is uniformly bounded. Consequently, using again Lemma ~\ref{linearalgebra}, the quantity
$$\sum_{j\ge 2}\int_{\bbR}|\partial_x\xi_j(x,t)|^2|\jap{R_t(\xi(x,t))^{-1}\bbP_{j}R_t(\xi(x,t))^{-1}X(t),\1}|^2\, dx$$
is bounded as $t\to \infty $. It remains to study the integral
$$\int_{\bbR}|\partial_x\xi_1(x,t)|^2|\jap{R_t(\xi (x,t))^{-1}\bbP_{1}R_t(\xi (x,t))^{-1}X(t) ,\1}|^2\, dx $$
which we minorize by  the integral $I(t)$ of the same function on an interval $J_t$ 
constructed as follows. Since $\lambda_1$ is a multiple eigenvalue, $1-\psi_1$ has at least one zero on the real line. Denote by $x_c$ such a zero. Since $\psi_1$ is a Blaschke product, we know that $i\psi_1'(x_c)$ is a non-zero real number. Consider the interval 
$$J_t=\left [x_c+\frac{\mu}{t}+\frac {\varkappa_1}{t^2}, x_c+\frac{\mu}{t}+\frac {\varkappa_2}{t^2}\right ]\ ,$$
where $\mu\not=0$ and $\varkappa_1<\varkappa_2$ are real numbers which we are going to choose. For $x\in J_t$, we can expand
\begin{align*}
\psi_1(x)&=1+\psi_1'(x_c)\left(\frac{\mu}{t}+\frac{\varkappa}{t^2}\right)+\frac12\psi_1''(x_c)\frac{\mu^2}{t^2}+O(t^{-3})
\\
&=1+\psi_1'(x_c)\frac{\mu}{t}+
\left(\psi_1'(x_c)\varkappa+\frac12\psi_1''(x_c)\mu^2\right)\frac1{t^2}+O(t^{-3})
\end{align*}
with $\varkappa_1\leq \varkappa\leq \varkappa_2$. It follows that 
\begin{align*}
\frac{1+\psi_1(x)}{1-\psi_1(x)}
&=
\frac{2+\psi_1'(x_c)\frac{\mu}{t}+O(t^{-2})}{-\psi_1'(x_c)\frac{\mu}{t}-(\psi_1'(x_c)\varkappa+\frac12\psi_1''(x_c)\mu^2)\frac1{t^2}+O(t^{-3})}
\\
&=-\frac{2t}{\psi_1'(x_c)\mu}
\frac{1+\frac12\psi_1'(x_c)\frac{\mu}{t}+O(t^{-2})}{1+(\frac{\varkappa}{\mu}+\frac{\mu}2\frac{\psi_1''(x_c)}{\psi_1'(x_c)})\frac1t+O(t^{-2})}
\\
&=-\frac{2t}{\psi_1'(x_c)\mu}
\left(1+\biggl(\frac{\mu}2\psi_1'(x_c)-\frac{\varkappa}{\mu}-\frac{\mu}{2}\frac{\psi_1''(x_c)}{\psi_1'(x_c)}\biggr)\frac1t+O(t^{-2})\right)
\\
&=-\frac{2}{\psi_1'(x_c)\mu}t
+\left(-1+\frac{2\varkappa}{\mu^2\psi_1'(x_c)}+\frac{\psi_1''(x_c)}{(\psi_1'(x_c))^2}  \right)+O(t^{-1})
\end{align*}
so that, in view of the expression \eqref{defb} of $b_1$, we obtain
\begin{align*}
\xi_1(x,t)=&\frac{\lambda_1^2}{2\pi}t-b_1(x)-\frac{x}{\nu_1^2}
\\
=&\frac{\lambda_1^2}{2\pi}t-\frac{i}{4\pi}\frac{\norm{1-\psi_1}^2}{\nu_1^2}\frac{1+\psi_1(x)}{1-\psi_1(x)}+\frac{\Re\jap{A_{\psi_1}(1-\psi_1),(1-\psi_1)}}{\nu_1^2\norm{1-\psi_1}^2}
\\
=&\left(\frac{\lambda_1^2}{2\pi}+\frac{i}{2\pi}\frac{\norm{1-\psi_1}^2}{\mu\nu_1^2\psi_1'(x_c)}\right)t
\\
&-\frac{i}{4\pi}\frac{\norm{1-\psi_1}^2}{\nu_1^2}
\left(-1+\frac{2\varkappa}{\mu^2\psi_1'(x_c)}+\frac{\psi_1''(x_c)}{(\psi_1'(x_c))^2}  \right)
\\
&-\frac{\Re\jap{A_{\psi_1}(1-\psi_1),(1-\psi_1)}}{\nu_1^2\norm{1-\psi_1}^2}+O(t^{-1})
\end{align*}
or equivalently
\begin{align*}
\xi_1(x,t)=& \zeta_1\, t +\eta_1(\varkappa)+O(t^{-1})\ ,\ x=x_c+\frac{\mu}{t}+\frac{\varkappa}{t^2}\ ,\ \varkappa \in [\varkappa_1,\varkappa_2]\ ,\\
\zeta_1:=&\frac{\lambda_1^2}{2\pi}+\frac{i}{2\pi}\frac{\norm{1-\psi_1}^2}{\mu\nu_1^2\psi_1'(x_c)}\ ,\\
\eta_1(\varkappa):=& 
-\frac{i}{4\pi}\frac{\norm{1-\psi_1}^2}{\nu_1^2}
\left(-1+\frac{2\varkappa}{\mu^2\psi_1'(x_c)}+\frac{\psi_1''(x_c)}{(\psi_1'(x_c))^2}  \right)
\\
&-\frac{\Re\jap{A_{\psi_1}(1-\psi_1),(1-\psi_1)}}{\nu_1^2\norm{1-\psi_1}^2}\ .
\end{align*}

We define $\mu $ such that $\zeta_1=0$. Consequently,  for $x$ in the interval $J_t$, 
$$\xi_1(x,t)=\eta _1(\varkappa)+ O\left (\frac 1t \right )\ .$$
 Let us write
$Z(x,t):=R_t(\xi (x,t))^{-1}X(t)\ ,\ \tilde Z(x,t):=R_t^*(\xi (x,t))^{-1}\1\ ,$ so that
$$\jap{R_t(\xi (x,t))^{-1}\bbP_{1}R_t(\xi (x,t))^{-1}X(t) ,\1}=Z_1(x,t)\overline {\tilde Z_1(x,t)}\ .$$
In order to estimate $Z_1,\tilde Z_1$ for $x\in J_t$, we write, by the definition of $Z$ and $\tilde Z$, 
$$\xi_1Z_1+[\mathcal A (t)^* Z]_1=\lambda_1e^{-i\varphi _1(t)}\ ,\ \xi_1\tilde Z_1+[\mathcal A (t) \tilde Z]_1=1\ .$$
Observe that \eqref{xij} implies  that, for $x\in J_t$ and for $j\ge 2$, $\xi_j(x,t)\sim c_jt$ for some $c_j>0$. Therefore, from  Lemma \ref{linearalgebra}, we infer
$$|Z_j(x,t)|+|\tilde Z_j(x,t)|\leq O(1/t)\ ,\ j\geq 2$$
and consequently,
$$\ Z_1(x,t)=z_1(\varkappa)e^{-i\varphi_1(t)}+O(1/t)\ ,\ \tilde Z_1(x,t)=\tilde z_1(\varkappa)+O(1/t)$$
where 
$$z_1(\varkappa):=\frac{\lambda_1}{\eta_1(\varkappa)+\frac{\overline \omega_1}{\lambda_1^2\nu_1^4}}\ ,\ 
\tilde z_1(\varkappa):=\frac{1}{\eta_1(\varkappa)+\frac{\omega_1}{\lambda_1^2\nu_1^4}}\ .$$
Here the parameters $\varkappa_1<\varkappa_2$ are chosen so that the denominators of $z_1(\varkappa)$ and $\tilde z_1(\varkappa)$ do not cancel for $\varkappa \in [\varkappa_1,\varkappa_2]$.
Consequently, for $x\in J_t$, $$|\jap{R_t(\xi (x,t))^{-1}\bbP_{1}R_t(\xi (x,t))^{-1}X(t),\1} |^2= |z_1(\varkappa)\tilde z_1(\varkappa)|^2 +O(1/t)\ $$
where $\varkappa :=t^2(x-x_c-\mu/t)\in [\varkappa_1,\varkappa_2]\ .$
On the other hand,
 $$|\partial_x\xi_1(x,t)|=\frac{\norm {1-\psi_1}^2}{2\pi \nu_1^2}\frac{|\psi'_1(x)|}{|1-\psi_1(x)|^2}\sim \frac{\norm {1-\psi_1}^2}{2\pi \nu_1^2|\psi_1'(x_c)|\mu ^2}t^2$$ for $x\in J_t$. Making the change of variable $\varkappa :=t^2(x-x_c-\mu/t)$ in the integral, we 
 we conclude that
\begin{eqnarray*}
I(t)&=&\int_{J_t} |\partial_x\xi_1(x,t)|^2|\jap{R_t(\xi (x,t))^{-1}\bbP_{1}R_t(\xi (x,t))^{-1}X(t) ,\1}|^2\, dx \\
&\sim &\frac{\norm {1-\psi_1}^4}{4\pi ^2\nu_1^4|\psi_1'(x_c)|^2\mu ^4}  \, t^2\int_{\varkappa_1}^{\varkappa_2}|z_1(\varkappa)\tilde z_2(\varkappa)|^2\, d\varkappa\ ,
\end{eqnarray*}
which completes the proof.
\end{proof}
\begin{remark*}
-- Proposition \ref{lowerbound} includes the case studied in \cite{OP2}, which corresponds to $N=1$ and was revisited in  Appendix B of \cite{GLPR} by solving explicitly the corresponding ODE system. Here our approach is more flexible so that we can deal with more general data, providing enough turbulent solutions to establish genericity in the next subsection.\\
-- In \cite{OP2}, it is proved that, if $u$ is a rational solution and if $H_u$ has only simple singular values, then all the Sobolev norms of $u$ stays bounded. Proposition \ref{lowerbound} shows that the situation may be dramatically different if there exists a multiple singular value for $H_u$. In fact, we expect that the existence of such a multiple singular value always implies that the norms of the solution in $W^{s,2}$ with large $s$ are unbounded.
\end{remark*}
\subsection{Lower bound for generic data in $W^{1,2}(\bbC_+)$}
 Denote by $\Phi(t)$  the flow map of the cubic Szeg\H{o} equation on $W^{1,2}(\bbC_+)$. 
 The main step in the proof of Theorem \ref{genericgrowth} is the following approximation result.
\begin{lemma}\label{approximation}
For every $u_0\in W^{1,2}(\bbC_+)$, there exists a family $(u_0^\e )$ in $W^{1,2}(\bbC_+)$ such that, as $\e \to 0$,
$$u_0^\e \to u_0\ {\rm in}\ W^{1,2}(\bbC_+)\ $$
and, for every $\e $,
$$\int_1^{+\infty}\frac{\norm{\partial_x\Phi (t)u_0^\e}_{L^2}}{t^2}\, dt=+\infty \ .$$
\end{lemma}
\begin{proof}
\emph{Step 1: reduction to rational $u_0$.}
Recall that  $W^{1,2}(\bbC_+)$ is a Hilbert space with the inner product
$$\jap{u,v}+\jap{\partial_xu,\partial_xv}\ .$$ 
Let $u_0\in W^{1,2}(\bbC_+)$. We first claim that $u_0$ can be approximated in $W^{1,2}(\bbC_+)$  by a sequence of rational functions. Indeed, the Fourier transform of a rational function in $W^{1,2}(\bbC_+)$ is a linear combination of 
$$\xi ^k{\rm e}^{-\alpha \xi }, \quad \xi>0,$$
where $k$ is a nonnegative integer and $\alpha $ is a complex number of positive real part.
By the Plancherel theorem, if $u\in W^{1,2}(\bbC_+)$ is orthogonal to all rational functions, then
$$\int_0^\infty (1+\xi^2)\hat u(\xi )\xi^k{\rm e}^{-\alpha \xi}\, d\xi =0$$
for  every complex number $\alpha $ with $\Re\alpha >0$. Consequently, by making $k=0$ and $\alpha $ tend to the imaginary axis, we infer that  $u=0$.
Therefore it is enough to prove the Lemma if $u_0$ is a rational function in $W^{1,2}(\bbC_+)$, which means
$$u_0(x)=\frac{A(x)}{B(x)}$$
where $B$ is a polynomial of degree $N\ge 1$ with zeros in the open lower half plane only, and 
$A$ is a polynomial of degree at most $N-1$, with no common factors with $B$. 

\emph{Step 2: reduction to $u_0$ with simple eigenvalues.}
For a given $N$, the set of rational functions $u_0$ as above is a complex manifold of dimension $2N$, on which 
the condition that $H_{u_0}^2$ has $N$ simple positive eigenvalues defines a dense open subset, characterised by 
$$\det (\langle H_{u_0}^{2(n+m)}u_0, u_0\rangle )_{0\leq n,m\leq N-1}\ne 0$$ 
(see  \cite{OP2} for more detail). So we are reduced to proving the statement for $u_0$ belonging to this dense open subset.

\emph{Step 3: defining $u_0^\eps$.}
Denote by $(\lambda_n, \psi_n, {\rm e}^{i\varphi_n}, \omega_n)_{1\leq n\leq N}$ the spectral data of $u_0$. 
Let $\varphi_{N+1}=0$, $\omega_{N+1}=i$ and $\lambda_{N+1}=\eps>0$; we also set 
$$
\psi_{N+1}(x)=\left(\frac{x-i}{x+i}\right)^2
$$
(although any Blaschke product of degree $\geq2$ will do). 
We define $u_0^\e $ to be the rational function with spectral data $(\lambda_n, \psi_n, {\rm e}^{i\varphi_n}, \omega_n)_{1\leq n\leq N+1}$. 
Our aim is to check that $u_0^\e \to u_0$ in $W^{1,2}(\bbC_+)$ as $\e \to 0$ by  applying the inverse spectral  formula \eqref{b25} of Theorem~\ref{thm.uniq}. But first we need to go through some preliminaries. 

Denote by $\mathcal A^\e $ the $(N+1)\times(N+1)$ matrix associated to $u_0^\e $, and by $\mathcal A$ the $N\times N$ matrix associated to $u_0$. 
In view of formulae \eqref{Aoff} and \eqref{Adiag}, we have
$$
\mathcal A_{N+1,N+1}^\varepsilon
=\frac{\omega_{N+1}}{\lambda_{N+1}^2\nu_{N+1}^4}
=\frac{\omega_{N+1}}{4\pi \Im \omega_{N+1}}=\frac{i}{4\pi},
$$
and 
$
\mathcal A_{N+1,k}^{\varepsilon}=\frac{i}{2\pi}+O(\varepsilon), 
\quad
\mathcal A_{k,N+1}^{\varepsilon}=O(\varepsilon), 
$
as $\eps\to0$ for any $k\leq N$. It follows that
$$
\mathcal A^\varepsilon=\mathcal A_0+X^\varepsilon, 
\quad
\mathcal A_0
=
\begin{pmatrix}
\calA & 0
\\
\frac{i}{2\pi}\jap{\cdot,\1_N} & \frac{i}{2\pi}
\end{pmatrix}, 
\quad
\norm{X^\eps}=O(\eps)
$$
as $\eps\to0$, where $\1_N=(1,\dots,1)\in\bbC^N$. 
Denote $\zeta\in\bbC_+^N$, $\zeta_{N+1}\in\bbC_+$ and $\wt\zeta=(\zeta,\zeta_{N+1})\in\bbC_+^{N+1}$.
By Theorem~\ref{thm.c1}, the inverse of 
$\calA^*-D(\zeta)$ exists and 
$$
\sup_{\zeta_1,\dots,\zeta_N\in\bbC_+}\norm{(\calA^*-D(\zeta))^{-1}}<\infty.
$$
We express the inverse of $\calA_0^*-D(\wt\zeta)$ as
\begin{equation}
(\calA_0^*-D(\wt\zeta))^{-1}
=
\begin{pmatrix}
(\calA^*-D(\zeta))^{-1} & \frac{i}{2\pi}(-\frac{i}{4\pi}-\zeta_{N+1})^{-1}(\calA^*-D(\zeta))^{-1}\1_N
\\
0 & (-\frac{i}{4\pi}-\zeta_{N+1})^{-1}
\end{pmatrix}\ .
\label{x0}
\end{equation}
From here we find that this inverse is uniformly bounded, 
$$
\sup_{\wt\zeta\in\bbC_+^{N+1}}\norm{(\calA_0^*-D(\wt\zeta))^{-1}}<\infty.
$$
From the resolvent identity
$$
((\calA^\eps)^*-D(\wt\zeta))^{-1}-(\calA_0^*-D(\wt\zeta))^{-1}
=
-((\calA^\eps)^*-D(\wt\zeta))^{-1}X^\eps(\calA_0^*-D(\wt\zeta))^{-1}
$$
and the estimate $\norm{X^\eps}=O(\eps)$ we find that the inverse of $(\calA^\eps)^*-D(\wt\zeta)$ is similarly uniformly bounded and moreover
\begin{equation}
\sup_{\wt\zeta\in\bbC_+^{N+1}}\norm{(\calA_0^*-D(\wt\zeta))^{-1}-((\calA^\eps)^*-D(\wt\zeta))^{-1}}=O(\eps)
\label{x00}
\end{equation}
as $\eps\to0$. 
Now let $Q(x)$ be the $N\times N$ matrix associated with $u_0$, and 
\begin{align*}
Q_\eps(x)&=(\calA^\eps)^*-xD(\nu^2)^{-1}-D(b(x)),
\\
Q_{\eps,0}(x)&=\calA_0^*-xD(\nu^2)^{-1}-D(b(x)).
\end{align*}
Since the eigenvalues $\lambda_1,\dots,\lambda_N$ are simple, we have $b_k(x)=0$ for $k=1\dots,N$ (see Remark~\ref{simplespec}), and from the explicit form of $\psi_{N+1}$ one obtains
$$
b_{N+1}(x)=-\frac1{\nu_{N+1}^2}\frac1{x}=-\frac{\eps}{\sqrt{4\pi}}\frac1x.
$$
Let $\bbP_{N+1}$ be the projection in $\bbC^{N+1}$ onto the subspace spanned by the last vector $(0,\dots,0,1)$ of the the canonical basis, and let $\bbP_{N+1}^\perp=I-\bbP_{N+1}$ be the projection onto the orthogonal subspace in $\bbC^{N+1}$. Denoting by $D(\lambda)$, $D(e^{i\varphi})$ the diagonal operators in $\bbC^{N+1}$, we have
\begin{align*}
u_0^{\eps}(x)&=\frac1{2\pi i}\jap{Q_\eps(x)^{-1}D(\lambda)D(e^{-i\varphi})\1,\1}_{\bbC^{N+1}},
\\
u_0(x)&=\frac1{2\pi i}\jap{Q(x)^{-1}\bbP_{N+1}^\perp D(\lambda)D(e^{-i\varphi})\1,\1}_{\bbC^{N}}.
\end{align*}

\emph{Step 4: Proof that $u_0^\eps\to u_0$ in $L^2(\bbR)$.}
We will check two facts:
\begin{align}
u_0^\eps(x)\to u_0(x) \text{ uniformly in $x\in\bbR$;}
\label{x1}
\\
\abs{u_0^\eps(x)}+\abs{u_0(x)}\leq \frac{C}{\abs{x}}, \quad \abs{x}\geq2,
\label{x2}
\end{align}
where the constant $C$ is independent of $\eps$.

Let us check \eqref{x1}. 
Since $\lambda_{N+1}=\eps$, we have 
\begin{equation}
u_0^\eps(x)
=\frac1{2\pi i}\jap{Q_\eps(x)^{-1}\bbP_{N+1}^\perp D(\lambda)D(e^{-i\varphi})\1,\1}_{\bbC^{N+1}}
+\frac{\eps}{2\pi i}\jap{Q_\eps(x)^{-1}\bbP_{N+1}\1,\1}_{\bbC^{N+1}},
\label{x3}
\end{equation}
where the second term in the r.h.s. is $O(\eps)$ uniformly in $x\in\bbR$. 
Using  \eqref{x00}, we replace $Q_\eps(x)^{-1}$ by $Q_{\eps,0}(x)^{-1}$ in the first term in the r.h.s., accruing a uniform $O(\eps)$ error. 
By the matrix structure in \eqref{x0}, we find
\begin{multline*}
\frac1{2\pi i}
\jap{Q_{\eps,0}(x)^{-1}\bbP_{N+1}^\perp D(\lambda)D(e^{-i\varphi})\1,\1}_{\bbC^{N+1}}
\\
=
\frac1{2\pi i}
\jap{Q(x)^{-1}\bbP_{N+1}^\perp D(\lambda)D(e^{-i\varphi})\1,\1}_{\bbC^{N}}=u_0(x).
\end{multline*}
This proves \eqref{x1}. 

Let us check \eqref{x2}. For $u_0(x)$ the estimate is obvious. For $u_0^\eps(x)$, we use the decomposition \eqref{x3} again. By the resolvent identity and the matrix structure in \eqref{x0}, 
$$
\norm{Q_\eps(x)^{-1}\bbP_{N+1}^\perp}
\leq
C\norm{Q_{\eps,0}(x)^{-1}\bbP_{N+1}^\perp}
\leq
C\norm{Q(x)^{-1}}\leq C/\abs{x},
$$
which gives the required estimate for the first term in the r.h.s. of \eqref{x3}. 
For the second term, for the same reasons, 
\begin{equation}
\eps\norm{Q_\eps(x)^{-1}\bbP_{N+1}}
\leq 
C\eps\norm{Q_{\eps,0}(x)^{-1}\bbP_{N+1}}
\leq 
C\eps\abs{(-\tfrac{i}{4\pi}-\tfrac{\eps}{\sqrt{4\pi}}(x-\tfrac1x))^{-1}}\leq C/\abs{x}
\label{x4}
\end{equation}
for $\abs{x}\geq2$. This concludes the proof of \eqref{x2}; we have checked that $u_0^\eps\to u_0$ in $L^2(\bbR)$. 

\emph{Step 5: proof that $\partial_x u_0^\eps\to \partial_x u_0$ in $L^2(\bbR)$.}
We have 
$$
\partial_x Q_\eps(x)=-D(\nu^2)^{-1}-D(b'(x)), 
$$
and so 
$$
\partial_x Q_\eps(x)^{-1}=Q_\eps(x)^{-1}D(\nu^2)^{-1}Q_\eps(x)^{-1}+\frac{\eps}{\sqrt{4\pi}}\frac1{x^2}Q_\eps(x)^{-1}\bbP_{N+1}Q_{\eps}(x)^{-1}, 
$$
because $b'_1(x)=\dots=b'_N(x)=0$ and $b'_{N+1}(x)=\frac{\eps}{\sqrt{4\pi}}\frac1{x^2}$. Thus, 
\begin{align*}
\partial_x u_0^{\eps}(x)
=&
\frac1{2\pi i}\jap{Q_\eps(x)^{-1}D(\nu^2)^{-1}Q_\eps(x)^{-1}D(\lambda)D(e^{-i\varphi})\1,\1}
\\
&+
\frac1{2\pi i}\frac{\eps}{\sqrt{4\pi}}\frac1{x^2}
\jap{Q_\eps(x)^{-1}\bbP_{N+1}Q_\eps(x)^{-1}D(\lambda)D(e^{-i\varphi})\1,\1}.
\end{align*}
In the same way as on the previous step of the proof, one proves that the first term in the r.h.s. here converges in $L^2(\bbR)$ to $\partial_x u_0(x)$. It remains to check that the second term converges to zero in $L^2(\bbR)$. 

Along with \eqref{x4}, and for the same reasons, we have the estimate
$$
\norm{\bbP_{N+1}Q_\eps(x)^{-1}}\leq C\abs{(-\tfrac{i}{4\pi}-\tfrac{\eps}{\sqrt{4\pi}}(x-\tfrac1x))^{-1}}.
$$
Using this, we find 
\begin{align*}
\frac{\eps}{x^2}&\abs{\jap{Q_\eps(x)^{-1}\bbP_{N+1}Q_\eps(x)^{-1}D(\lambda)D(e^{-i\varphi})\1,\1}}
\\
&\lesssim
\frac{\eps}{x^2}\norm{Q_\eps(x)^{-1}\bbP_{N+1}Q_\eps(x)^{-1}D(\lambda)D(e^{-i\varphi})}
\\
&\lesssim
\frac{\eps}{x^2}
\abs{(-\tfrac{i}{4\pi}-\tfrac{\eps}{\sqrt{4\pi}}(x-\tfrac1x))^{-1}}
\norm{\bbP_{N+1}Q_\eps(x)^{-1}D(\lambda)D(e^{-i\varphi})}
\\
&\lesssim
\frac{\eps}{x^2}
\abs{(-\tfrac{i}{4\pi}-\tfrac{\eps}{\sqrt{4\pi}}(x-\tfrac1x))^{-1}}
\bigl(
\norm{\bbP_{N+1}Q_\eps(x)^{-1}\bbP_{N+1}^\perp}
+
\eps\norm{\bbP_{N+1}Q_\eps(x)^{-1}\bbP_{N+1}}
\bigr).
\end{align*}
Further, by \eqref{x0} we have 
$$
\bbP_{N+1}Q_{\eps,0}(x)^{-1} \bbP_{N+1}^\perp=0
$$
and so, by the resolvent identity,
\begin{align*}
\norm{\bbP_{N+1}Q_{\eps}(x)^{-1} \bbP_{N+1}^\perp}
&=
\norm{\bbP_{N+1}Q_{\eps,0}(x)^{-1}(X^\eps)^*Q_{\eps}(x)^{-1}\bbP_{N+1}^\perp}
\\
&\lesssim
\norm{\bbP_{N+1}Q_{\eps,0}(x)^{-1}}\norm{X^\eps}
\lesssim
\eps \abs{(-\tfrac{i}{4\pi}-\tfrac{\eps}{\sqrt{4\pi}}(x-\tfrac1x))^{-1}}.
\end{align*}
Putting this together, we find
$$
\frac{\eps}{x^2}\abs{\jap{Q_\eps(x)^{-1}\bbP_{N+1}Q_\eps(x)^{-1}D(\lambda)D(e^{-i\varphi})\1,\1}}
\lesssim
\frac{\eps^2}{x^2}
\abs{(-\tfrac{i}{4\pi}-\tfrac{\eps}{\sqrt{4\pi}}(x-\tfrac1x))^{-2}}.
$$
Now it's a matter of elementary calculation to check that the function in the r.h.s. converges to zero in $L^2(\bbR)$. For example, it is easy to see that this function is bounded by $C\min\{1,\eps^2/x^4\}$, which yields the required convergence. We have checked that $\partial_x u_0^\eps\to \partial_x u_0$ in $L^2(\bbR)$, and so $u_0^\e \to u_0$ in $W^{1,2}(\bbC_+)$. 

\emph{Step 6: concluding the proof.}
From Proposition \ref{lowerbound}, we know that
$$\liminf _{t\to +\infty}\frac{\norm{\partial_x\Phi (t)u_0^\e}_{L^2}}{t}>0$$
for every $\e >0$, 
and the lemma is proved.
\end{proof} 
Let us complete the proof of Theorem \ref{genericgrowth}. For every positive integer $n$, we consider
$$\Omega _n:=\left \{ u_0\in W^{1,2}(\bbC_+): \exists t_n, \int_1^{t_n}\frac{\norm{\partial_x\Phi (t)u_0^\e}_{L^2}}{t^2}\, dt>n\ \right \}\ .$$
By the wellposedness of the cubic Szeg\H{o} equation on $W^{1,2}(\bbC_+)$ (see \cite{OP2}), the map  $$u_0\in W^{1,2}(\bbC_+)\mapsto \Phi (\cdot)u_0\in C([0,T], W^{1,2}(\bbC_+))$$ is continuous for every $T>0$, and therefore $\Omega_n$ is an open subset of $W^{1,2}(\bbC_+)$. Furthermore, by Lemma \ref{approximation}, $\Omega_n$ is dense in $W^{1,2}(\bbC_+)$. Hence Baire's theorem implies that 
$$G:=\bigcap_{n\ge 1}\Omega_n$$
is a dense $G_\delta $ subset of $W^{1,2}(\bbC_+)$. This completes the proof of Theorem~\ref{genericgrowth}.

\section*{Acknowledgements}
The authors are grateful to V.~Kapustin and O.~Pocovnicu for useful discussions. 
A.P. is grateful to the Department of Mathematics, Universit\'e Paris-Saclay, for hospitality. 
A.P. was supported by the Ministry of Science and Higher Education of the Russian Federation, contract No. 075-15-2019-1619.

\end{document}